%% file: ams.tex
\documentclass[oneside]{amsart}
\usepackage[T1]{fontenc}
\usepackage[latin9]{inputenc}
\usepackage{verbatim}
\usepackage{float}
\usepackage{bm}
\usepackage{amstext}
\usepackage{amsthm}
\usepackage{amssymb}
\usepackage{graphicx}

\makeatletter
\numberwithin{equation}{section}
\numberwithin{figure}{section}

\theoremstyle{plain}
\newtheorem{thm}{Theorem}[section]
  \newtheorem*{question*}{Question}
  \newtheorem*{conjecture*}{Conjecture}
  \newtheorem{prop}[thm]{Proposition}
  \newtheorem{lyxalgorithm}[thm]{Algorithm}

  \theoremstyle{definition}
  \newtheorem{example}[thm]{Example}
  \newtheorem{defn}[thm]{Definition}

  \theoremstyle{remark}
  \newtheorem*{acknowledgement*}{Acknowledgement}

\usepackage{amsbsy}
\usepackage{hyperref}
\usepackage{breqn}

\allowdisplaybreaks[4]

\makeatother

\newcommand{\sep}{, }

\begin{document}
	\input{definitions.tex}

	\input{abstract.tex}
	\input{bio.tex}

	\input{acknowledgements.tex}

	\title[\expandafter\uppercase\expandafter{\ShortPaperTitle}]{\expandafter\uppercase\expandafter{\PaperTitle}}
	\author{\Miek}
	\address{\Miek\UPAddress}
	\email{\MiekEmail}
	\begin{abstract}%
	\paperabstract %	
	\end{abstract}

	\subjclass[2010]{Primary: \MSCCodesPrimary. Secondary: \MSCCodesSecondary}

	\maketitle

	\input{content}

	\begin{acknowledgement*}
	\PeopleAck

	\ClaudeLeonAck
	\end{acknowledgement*}

	\bibliographystyle{plain}
	\bibliography{bib}

\end{document}

%% file: definitions.tex
\global\long\def\rad{\text{radii}}

\global\long\def\radius{\text{radius}}

\global\long\def\set#1#2{\left\{  \vphantom{#1}\vphantom{#2}#1\right.\left|\ #2\vphantom{#1}\vphantom{#2}\right\}  }

\global\long\def\petalangle#1#2#3{\theta(#1,#2,#3)}

\global\long\def\a{\bm{\alpha}}

\global\long\def\b{\bm{\beta}}

\global\long\def\c{\bm{\gamma}}

\global\long\def\t{\bm{\tau}}

\global\long\def\petalangvec#1{\overline{\alpha}^{(\mathbf{#1})}}

\global\long\def\petalnumvec#1{\overline{\xi}^{(#1)}}

\global\long\def\Nzero{\mathbb{N}_{0}}

\global\long\def\tuples{\mathbb{T}}

\global\long\def\N{\mathbb{N}}

\global\long\def\R{\mathbb{R}}

\global\long\def\Q{\mathbb{Q}}

\global\long\def\Z{\mathbb{Z}}

\global\long\def\duality#1#2{\left\langle \vphantom{#1}\vphantom{#2}#1\right.\left|\ #2\vphantom{#1}\vphantom{#2}\right\rangle }

\global\long\def\dot#1#2{#1\cdot#2}

\global\long\def\parenth#1{\left(#1\right)}

\global\long\def\abs#1{\left|#1\right|}

\global\long\def\snonhex#1{\text{s-NonHex}\left(#1\right)}

\global\long\def\rnonhex#1{\text{r-NonHex}\left(#1\right)}

\global\long\def\ononhex#1{\text{1-NonHex}\left(#1\right)}

\global\long\def\sbounds#1{\text{s-Bounds}\left(#1\right)}

\global\long\def\rbounds#1{\text{r-Bounds}\left(#1\right)}

\global\long\def\rfewlargeneighbors#1{\text{r-FewLargeNeighbors}\left(#1\right)}

\global\long\def\rboundsextra#1{\text{r-BoundsExtra}\left(#1\right)}

\global\long\def\srdisjunct#1{\text{sr-Disjunct}\left(#1\right)}

\global\long\def\seq#1{\text{Seq}\left(#1\right)}

\global\long\def\modtwo#1{\text{Mod2}\left(#1\right)}

\global\long\def\snec#1{\text{s-Necessary}\left(#1\right)}

\global\long\def\rnec#1{\text{r-Necessary}\left(#1\right)}

\global\long\def\onec#1{\text{1-Necessary}\left(#1\right)}

\global\long\def\rverticalcont#1{\text{r-VerticalContour}\left(#1\right)}

%% file: abstract.tex
\newcommand{\ShortPaperTitle}{Compact 3-packings of the plane}
\newcommand{\PaperTitle}{On compact packings of the plane\\ with circles of three radii}

\newcommand{\paperabstract}{%
A compact circle-packing $P$ of the Euclidean plane is a set of circles
which bound mutually disjoint open discs with the property that, for
every circle $S\in P$, there exists a maximal indexed set $\{A_{0},\ldots,A_{n-1}\}\subseteq P$
so that, for every $i\in\{0,\ldots,n-1\}$, the circle $A_{i}$ is
tangent to both circles $S$ and $A_{i+1\mod n}.$ 

We show that there exist at most $13617$ pairs $(r,s)$ with $0<s<r<1$
for which there exist a compact circle-packing of the plane consisting
of circles with radii $s$, $r$ and $1$.

We discuss computing the exact values of such $0<s<r<1$ as roots
of polynomials and exhibit a selection of compact circle-packings
consisting of circles of three radii. We also discuss the apparent
infeasibility of computing \emph{all }these values on contemporary
consumer hardware.  }

\newcommand{\MSCCodesPrimary}{%
	52C15% 	Packing and covering in $2$ dimensions 
}%
\newcommand{\MSCCodesSecondary}{% 	
 	68U05\sep % Computer graphics; computational geometry
	05A99% Combinatorics,  	None of the above, but in this section
}%
\newcommand{\paperkeywords}{compact circle packing\sep three-packings} %

%% file: bio.tex
\newcommand{\Title}{}

\newcommand{\Miek}{Miek Messerschmidt}

\newcommand{\MiekEmail}{mmesserschmidt@gmail.com}

\newcommand{\UPAddress}{Department of Mathematics and Applied Mathematics; University of Pretoria; Private~bag~X20 Hatfield; 0028 Pretoria; South Africa}

\newcommand{\UPAddresswithbreaks}{%
Department of Mathematics and Applied Mathematics\\
University of Pretoria\\
Private~bag~X20\\
Hatfield\\
0028 Pretoria\\
South Africa}

%% file: acknowledgements.tex
\newcommand{\PeopleAck}{The author would like to express his thanks to Frits Veerman, Eder
Kikianty and Janko B\"ohm for helpful conversations, and to the University
of Kaiserslautern which graciously provided some computational resources
to the author.}

\newcommand{\ClaudeLeonAck}{The author's research was funded by the Claude Leon Foundation.}

%% file: content.tex
\section{Introduction}

By a \emph{circle-packing }(or just \emph{packing}) $P$ we mean a
set of circles in the Euclidean plane, so that the open discs bounded
by the circles are pairwise disjoint. We define $\rad(P):=\set{\radius(S)}{S\in P}$.
If $|\rad(P)|<\infty$, we will assume that $P$ is maximal and that
it is scaled so that $\max\rad(P)=1$. For $n\in\N$, we will say
$P$ is an \emph{$n$-packing} if $|\rad(P)|=n$. We say a circle-packing
$P$ is \emph{compact} if, for every circle $S\in P$, there exists
some $m\in\N$ and a maximal indexed set of circles $\{A_{0},\ldots,A_{m-1}\}\subseteq P$
so that all the circles $A_{0},\ldots,A_{m-1}$ are tangent to $S$
and, for every $i\in\{0,\ldots,m-1\}$, the circle $A_{i}$ is tangent
to $A_{i+1\mod m}$. The circles $A_{0},\ldots,A_{m-1}$ are called
the \emph{neighbors} of $S$. For $n\in\N$, we define the sets
\[
\Delta_{n}:=\set{(r_{i})_{i=1}^{n-1}\in(0,1)^{n-1}}{0<r_{n-1}<\ldots<r_{1}<1}
\]
 and 
\[
\Pi_{n}:=\set{(r_{i})\in\Delta_{n}}{\begin{array}{c}
\text{There exists a compact \ensuremath{n}-packing \ensuremath{P}}\\
\text{with }\rad(P)=\{r_{1},\ldots,r_{n-1},1\}.
\end{array}}.
\]

\begin{figure}[b]
	\fbox{\includegraphics[clip,width=1\textwidth]{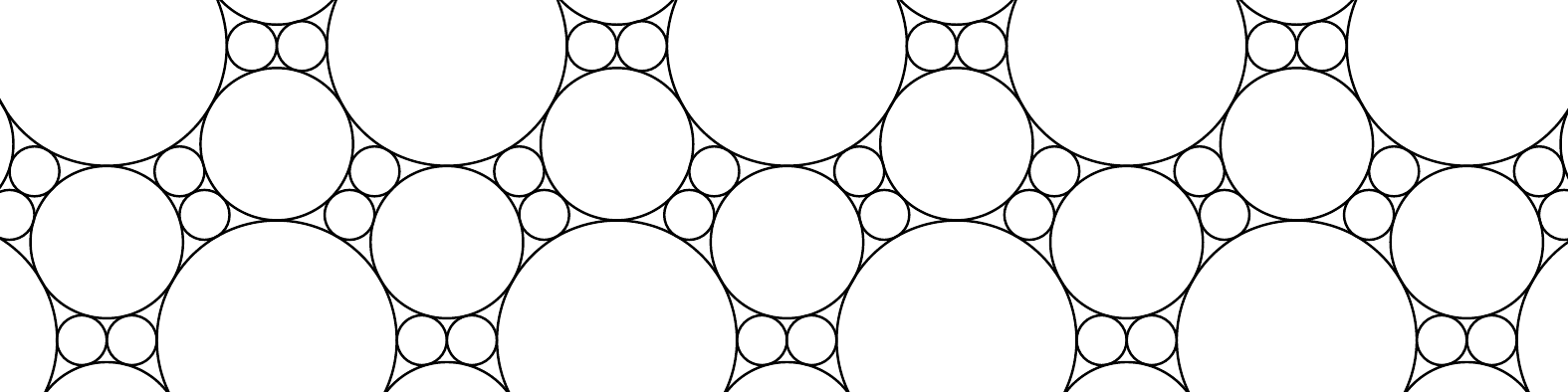}}\caption{\label{fig:first-non-trivial-example}
	A compact circle-packing $P$
	with $\protect\rad(P)=\{s_{0},r_{0},1\}$ where $s_{0}\approx0.208266$
	and $r_{0}\approx0.635671$ are respective roots of the polynomials
	$1+2s-27s^{2}-28s^{3}+4s^{4}$ and $-1-12r-18r^{2}+60r^{3}+3r^{4}$.
	We have $(r_{0},s_{0})\in\Pi_{3}$, but there exists no compact circle-packing
	$Q$ with $\protect\rad(Q)=\{r_{0},1\}$, \cite{Kennedy}.}
	\end{figure}

In \cite{Kennedy}, Kennedy proved that $|\Pi_{2}|=9$, i.e., that
there exist exactly nine values of $r_{0}\in(0,1)$ for which there
exist compact $2$-packings $P$ with $\rad(P)=\{r_{0},1\}$. Eight
of these nine values were known previously: seven values appear in
\cite{Toth} and a further one in \cite{LikosHenley}. Kennedy computed
the remaining value\footnote{A root of the polynomial $r^{8}-8r^{7}-44r^{6}-232r^{5}-482r^{4}-24r^{3}+388r^{2}-120r+9$.}
$r_{0}\approx0.545151$ and demonstrated the existence of a compact
$2$-packing $P$ with $\rad(P)=\{r_{0},1\}$.\medskip

In this paper we will concern ourselves with compact $3$-packings. 

Of course, one may construct a compact $3$-packing by packing circles
into the interstitial gaps of a compact $2$-packing, hence $|\Pi_{3}|\geq|\Pi_{2}|=9$.
Therefore the first question to ask is whether there exists a compact
$3$-packing that does not arise in this way. By merely guessing,
it is possible to construct such a packing, cf. Figure~\ref{fig:first-non-trivial-example}.
Another such packing appears in \cite[Fig. 15. 27/1, p.187]{Toth}.
Hence $|\Pi_{3}|>|\Pi_{2}|=9$ and since not all compact $3$-packings
arise from compact $2$-packings by filling interstitial gaps, we
are motivated to ask:
\begin{question*}
What is the cardinality of $\Pi_{3}$?
\end{question*}
The first goal of this paper is to answer this question by proving
that $\Pi_{3}$ is finite (cf. Theorem~\ref{thm:set-radii-admitting-3packings-are-finite}).
The second goal is to obtain the bound $|\Pi_{3}|\leq13617$ (cf.
Sections~\ref{sec:Numerical-computation-of-upperbound-of-Pi3} and~\ref{sec:exact-values-for-L}).

We briefly describe the analysis leading up to this result. 

The main idea follows the arguments presented in \cite{Kennedy} quite
closely in spirit, but does become more technical and relies significantly
on searches performed by computer. The majority of the work concerns
analysis of the functions $\a,\b,\c:\Delta_{3}\to(0,\pi)^{6}$ (these
functions are defined explicitly in Section~\ref{sec:Preliminaries}),
which parameterize the possible sizes of angles formed by connecting
the centers of mutually tangent circles of radii $s$, $r$ or $1$
with $0<s<r<1$. By construction (cf. Section~\ref{sec:Preliminaries}),
a necessary condition for $(r_{0},s_{0})\in\Delta_{3}$ to be an element
of $\Pi_{3}$ is that there exist specific tuples $\eta,\zeta,\xi\in\Z^{6}$
with non-negative coordinates satisfying 
\[
\eta\cdot\a(r_{0},s_{0})=\zeta\cdot\b(r_{0},s_{0})=\xi\cdot\c(r_{0},s_{0})=2\pi.
\]
I.e., the $2\pi$-contours of the three functions $\Delta_{3}\ni(r,s)\mapsto\eta\cdot\a(r,s)$,
$\Delta_{3}\ni(r,s)\mapsto\zeta\cdot\b(r,s)$ and $\Delta_{3}\ni(r,s)\mapsto\xi\cdot\b(r,s)$
intersect in $(r_{0},s_{0})$.

Theorem~\ref{thm:necessary-conditions-on-r-s-tuples} establishes
necessary conditions that such tuples $\eta,\zeta,\xi\in\Z^{6}$ must
satisfy, and one easily computes that there exist only 55 tuples $\eta$
for which it is possible to have $\eta\cdot\a(r_{0},s_{0})=2\pi$
for some $(r_{0},s_{0})\in\Pi_{3}$, (cf. Proposition~\ref{prop:55-s-angle-counts}).
In Section~\ref{sec:Contour-analysis}, by a careful and rather technical
analysis of the $2\pi$-contours of the functions $\Delta_{3}\ni(r,s)\mapsto\eta\cdot\a(r,s)$
and $\Delta_{3}\ni(r,s)\mapsto\zeta\cdot\b(r,s)$ and using the 55
tuples $\eta$ computed earlier, we establish a final necessary condition
on $\zeta\in\Z^{6}$ for $\zeta\cdot\b(r_{0},s_{0})=2\pi$ to hold
for some $(r_{0},s_{0})\in\Pi_{3}$. This final condition shows that
there can exist only finitely many such $\zeta\in\Z^{6}$, and allows
for the exact computation of all $248395$ elements of a certain set
$K\subseteq\Z^{6\times2}$ consisting of all tuples $\eta$ and $\zeta\in\Z^{6}$
which satisfy the necessary conditions that we established (cf. Proposition~\ref{prop:tuples-satisfying-necessary-conditions-are-finite}).
By observing that each element of $\Pi_{3}$ is determined by an element
from $K$ (cf. Proposition~\ref{prop:final-necessary-condition-for-3-packing}),
and that each element of $K$ determines at most one element of $\Pi_{3}$
(cf. Proposition~\ref{prop:unique-intercepts}), we conclude that
the set $\Pi_{3}$ is finite and that $|\Pi_{3}|\leq248395$ (cf.
Theorem~\ref{thm:set-radii-admitting-3packings-are-finite}). 

A further analysis of the $2\pi$-contours of the functions $\Delta_{3}\ni(r,s)\mapsto\eta\cdot\a(r,s)$
and $\Delta_{3}\ni(r,s)\mapsto\zeta\cdot\b(r,s)$ in Section~\ref{sec:Necessary-and-sufficient-conditions-for-contour-intercepts},
provides necessary and sufficient conditions for such contours to
intersect and allow for determining the sharper bound $|\Pi_{3}|\leq13617$
with methods described in Sections~\ref{sec:Numerical-computation-of-upperbound-of-Pi3}
and~\ref{sec:exact-values-for-L}.

The results of our computations are included as a dataset.

From here, further excluding elements that are not in $\Pi_{3}$ exactly
seems to be infeasible on contemporary consumer hardware. Firstly,
for a candidate element $(r_{0},s_{0})\in\Delta_{3}$ to be an element
of $\Pi_{3}$, there necessarily must exist a certain tuple $\xi\in\Z^{6}$
that satisfies $\xi\cdot\c(r_{0},s_{0})=2\pi$. However, the search
space of all $\xi\in\Z^{6}$ that might satisfy $\xi\cdot\c(r_{0},s_{0})=2\pi$
can sometimes be very large (containing up to $7\times10^{21}$ elements),
and is hence very time-consuming to sift through (cf. Section~\ref{sec:Numerical-computation-of-upperbound-of-Pi3}).
Secondly, given a numerical approximation of a candidate element $(r_{0},s_{0})$
of $\Pi_{3}$, it is possible to compute polynomials which have the
exact values of $r_{0}$ and $s_{0}$ as roots (cf. Section~\ref{sec:exact-values-for-L}).
Although computing these polynomials proceeds through a simple algorithm
(Algorithm~\ref{alg:detrig}) together with computing standard Gr\"obner
bases, performing the actual computation can be very time-consuming
and RAM intensive depending on the input. 

It is however possible to compute certain elements of $\Pi_{3}$ exactly,
and we display an arbitrary (but far from exhaustive) selection of
compact $3$-packings in the final section.

\medskip

The fact that both $\Pi_{2}$ and $\Pi_{3}$ are finite, and that
not every compact $3$-packing arises from a compact $2$-packing
by filling interstitial gaps, motivates the following conjecture:
\begin{conjecture*}
For every $n\in\N$, the set $\Pi_{n}$ is finite and the sequence
$(|\Pi_{n}|)_{n\in\N}$ is strictly increasing. 

Furthermore, for every $n\in\N$, there exists an element $(r_{n-1},r_{n-2},\ldots,r_{1})\in\Pi_{n}$
with $(r_{n-2},\ldots,r_{1})\notin\Pi_{n-1}$, i.e., not every compact
$n$-packing arises from filling interstitial gaps of a compact $(n-1)$-packing.
\end{conjecture*}
The computational nature of the problem might be a hindrance to proving
this conjecture. Proposition~\ref{prop:55-s-angle-counts}, which
is established purely by exhaustion performed by computer, is required
to bootstrap the proofs of Propositions~\ref{prop:s-contour-analysis}(\ref{enu:s-contour-global-bound})
and~\ref{prop:r-contour-analysis}(\ref{enu:r-contour-global-bound})
which together are crucial in our proof that $|\Pi_{3}|<\infty$.
This suggests that establishing $|\Pi_{n}|<\infty$ for specific values
of $n\in\N$ might be easier, but perhaps less interesting, than a
proof of the above conjecture in its full generality.

\section{Preliminary definitions, results and notation\label{sec:Preliminaries}}

We define $\Nzero:=\N\cup\{0\}$ and $\tuples:=\Nzero^{6}$. Let three
mutually tangent circles $A,B$ and $C$ have respective radii $a,b$
and $c$. By the cosine rule, the angle $\petalangle abc$ formed
at the center of $A$ by the line segments connecting the center of
$A$ with the centers of $B$ and $C$ is given by 
\[
\petalangle abc=\arccos\parenth{\frac{(a+b)^{2}+(a+c)^{2}-(b+c)^{2}}{2(a+c)(a+b)}}.
\]
We define the functions $\a,\b,\c:\Delta_{3}\to(0,\pi)^{6}$, for
$(r,s)\in\Delta_{3}$ by 
\begin{align*}
\a(r,s) & :=(\petalangle s11,\petalangle srr,\petalangle sss,\petalangle s1r,\petalangle s1s,\petalangle srs),\\
\b(r,s) & :=(\petalangle r11,\petalangle rrr,\petalangle rss,\petalangle r1r,\petalangle r1s,\petalangle rrs),\\
\c(r,s) & :=(\petalangle 111,\petalangle 1rr,\petalangle 1ss,\petalangle 11r,\petalangle 11s,\petalangle 1rs).
\end{align*}

Let $(r_{0},s_{0})\in\Pi_{3}$ be fixed and let $P$ be a compact
$3$-packing with $\rad(P)=\{s_{0},r_{0},1\}$. For any $D\in P$
of radius $t\in\{s_{0},r_{0},1\}$. We set 
\[
\bm{\tau}:=\begin{cases}
\a & t=s_{0}\\
\b & t=r_{0}\\
\c & t=1.
\end{cases}
\]
Let $\{A_{0},\ldots,A_{n-1}\}\subseteq P$ be the sequence of neighbors
of $D$ for some $n\in\N$. Connecting the center of $D$ with the
centers of $A_{0},\ldots,A_{n-1}$, we denote the \emph{angle-count
for $D$ }by $\xi^{(D)}\in\tuples$, which has, for $i\in\{1,2,3,4,5,6\}$,
its $i$'th coordinate defined as number of times the angle $\bm{\tau}_{i}(r_{0},s_{0})$
occurs around the center of $D$. Explicitly: We define 
\[
\begin{array}{rclcrcrcl}
\kappa_{1,1} & := & (1,0,0,0,0,0) & \quad & \kappa_{1,r_{0}} & := & \kappa_{r_{0},1} & := & (0,0,0,1,0,0)\\
\kappa_{r_{0},r_{0}} & := & (0,1,0,0,0,0) &  & \kappa_{1,s_{0}} & := & \kappa_{s_{0},1} & := & (0,0,0,0,1,0)\\
\kappa_{s_{0},s_{0}} & := & (0,0,1,0,0,0) &  & \kappa_{r_{0},s_{0}} & := & \kappa_{s_{0},r_{0}} & := & (0,0,0,0,0,1),
\end{array}
\]
and define $\sigma(j):=\radius(A_{j})\in\{s_{0},r_{0},1\}$ for $j\in\{0,\ldots,n-1\}$.
Then 
\[
\xi^{(D)}:=\sum_{j=0}^{n-1}\kappa_{\sigma(j),\sigma(j+1\mod n)}.
\]
Since $\{A_{0},\ldots,A_{n-1}\}$ is the sequence of neighbors of
$D$, it is clear that 
\[
\xi^{(D)}\cdot\bm{\tau}(r_{0},s_{0})=2\pi.
\]
Hence, for all circles $A$, $B$ and $C$ in $P$ with respective
radii $s_{0},r_{0}$ and $1$, we have 
\[
\xi^{(A)}\cdot\a(r_{0},s_{0})=\xi^{(B)}\cdot\b(r_{0},s_{0})=\xi^{(C)}\cdot\c(r_{0},s_{0})=2\pi.
\]
Therefore, a necessary condition for $(r_{0},s_{0})$ to be an element
of $\Pi_{3}$ is that the $2\pi$-contours of the functions
\begin{align*}
 & \Delta_{3}\ni(r,s)\mapsto\xi^{(A)}\cdot\a(r,s),\\
 & \Delta_{3}\ni(r,s)\mapsto\xi^{(B)}\cdot\b(r,s),\\
 & \Delta_{3}\ni(r,s)\mapsto\xi^{(C)}\cdot\c(r,s),
\end{align*}
intersect in $(r_{0},s_{0})$. 

\section{\label{sec:Necessary-conditions-on-angle-counts}Necessary conditions
on angle-counts}

Let $(s_{0},r_{0})\in\Pi_{3}$ and let $P$ be a compact $3$-packing
with $\rad(P)=\{s_{0},r_{0},1\}$. In the current section we show
that there necessarily exist circles $A$,$B$ and $C$ in $P$ of
respective radii $s_{0}$,$r_{0}$ and $1$, whose angle counts $\xi^{(A)}$,
$\xi^{(B)}$ and $\xi^{(C)}$ must necessarily satisfy certain conditions.
These conditions are collected in Theorem~\ref{thm:necessary-conditions-on-r-s-tuples}.
This theorem prompts the definition of a number of predicates in Definition~\ref{def:predicates}
which can be easily implemented in any programming language. Finally,
Proposition~\ref{prop:55-s-angle-counts} lists all 55 possible values
that the tuple $\xi^{(A)}$ can take on. This observation is crucial
in later sections for showing that the tuple $\xi^{(B)}$ may also
only take on finitely many values. 

We begin by observing that for each radius $t\in\{s_{0},r_{0},1\}$,
there must exist a circle $D\in P$ of $t$ that is not fully surrounded
by $6$ neighbors with radius $t$. 
\begin{prop}
\label{prop:not-in-center-of-hexagon}Let $(r_{0},s_{0})\in\Pi_{3}$
and let $P$ be any compact $3$-packing with $\rad(P)=\{s_{0},r_{0},1\}$. 
\begin{enumerate}
\item There exists a circle $A\in P$ with radius $s_{0}$ so that $\xi^{(A)}\cdot(1,1,0,1,1,1)\neq0$
and $\xi^{(A)}\cdot(0,0,1,0,0,0)<6$.
\item There exists a circle $B\in P$ with radius $r_{0}$ so that $\xi^{(B)}\cdot(1,0,1,1,1,1)\neq0$
and $\xi^{(B)}\cdot(0,1,0,0,0,0)<6$. 
\item There exists a circle $C\in P$ with radius $1$ so that $\xi^{(C)}\cdot(0,1,1,1,1,1)\neq0$
and $\xi^{(C)}\cdot(1,0,0,0,0,0)<6$.
\end{enumerate}
\end{prop}

\begin{proof}
We prove (1). Suppose, for all $A\in P$ of radius $s_{0}$, that
$\xi^{(A)}\cdot(1,1,0,1,1,1)=0$. Then we must have $\xi^{(A)}\cdot(0,0,1,0,0,0)=6$
for all $A\in P$ of radius $s_{0}$, and hence $P$ either cannot
contain circles of radius $s_{0}$, or cannot contain circles of radii
$1$ or $r_{0}$. This contradicts $\rad(P)=\{s_{0},r_{0},1\}$. The
other assertions follow similarly.
\end{proof}
Next, we observe that there must exist a pair of circles of respective
radii $s_{0}$ or $r_{0}$, so that at least one of these circles
has a neighbor of radius $1$.
\begin{prop}
\label{prop:either-s-or-r-circle-must-touch-1-circle}Let $(r_{0},s_{0})\in\Pi_{3}$
and let $P$ be any compact $3$-packing with $\rad(P)=\{s_{0},r_{0},1\}$.
There exist circles $A$ and $B$ from $P$ with respective radii
$s_{0}$ and $r_{0}$, so that $\xi^{(A)}\cdot(1,0,0,1,1,0)\neq0$
or $\xi^{(B)}\cdot(1,0,0,1,1,0)\neq0$.
\end{prop}

\begin{proof}
If it were the case that for every pair of circles $A$ and $B$ from
$P$ with respective radii $s_{0}$ and $r_{0}$ that $\xi^{(A)}\cdot(1,0,0,1,1,0)=0$
and $\xi^{(B)}\cdot(1,0,0,1,1,0)=0$, then $P$ could not contain
any circles of radius~$1$, or consists only of circles of radius~$1$.
This contradicts $\rad(P)=\{s_{0},r_{0},1\}$.
\end{proof}
Propositions~\ref{prop:mod-2-condition} through~\ref{prop:Few-large-neighbors},
establishes general necessary conditions that circles in $P$ must
satisfy.
\begin{prop}
\label{prop:mod-2-condition}Let $P$ be any compact $3$-packing.
For every circle $D\in P$, we have 
\begin{enumerate}
\item $\xi^{(D)}\cdot(2,0,0,1,1,0)=0\mod2.$
\item $\xi^{(D)}\cdot(0,2,0,1,0,1)=0\mod2.$
\item $\xi^{(D)}\cdot(0,0,2,0,1,1)=0\mod2.$
\end{enumerate}
\end{prop}

\begin{proof}
Let $(r_{0},s_{0})\in\Pi_{3}$ be such that $\rad(P)=\{s_{0},r_{0},1\}$.
Consider any circle $C$ in the packing $P$ of radius $1$. The line
segments connecting the center of $C$ to the centers of its neighboring
circles in $P$, can only have lengths $2$, $1+r_{0}$ or $1+s_{0}$.
Since each such line segment is a leg of exactly two angles formed
around the center of $C$, the angle-count with a given leg-length
must be even. I.e., 
\begin{align*}
\xi^{(C)}\cdot(2,0,0,1,1,0) & =0\mod2,\\
\xi^{(C)}\cdot(0,2,0,1,0,1) & =0\mod2,\\
\xi^{(C)}\cdot(0,0,2,0,1,1) & =0\mod2.
\end{align*}
A similar argument will establish the result if $C$ has radius $r_{0}$
or $s_{0}$.
\end{proof}
\begin{prop}
\label{prop:s-upper-lower-bounds}Let $(r_{0},s_{0})\in\Pi_{3}$ and
let $P$ be any compact $3$-packing with $\rad(P)=\{s_{0},r_{0},1\}$.
For any circle $A\in P$ of radius $s_{0}$ with $\xi^{(A)}\cdot(1,1,0,1,1,1)\neq0$,
we have $\xi^{(A)}\cdot(1,1,1,1,1,1)<6$ and $\xi^{(A)}\cdot(6,6,2,6,3,3)>12$.
\end{prop}

\begin{proof}
Since $0<s_{0}<r_{0}<1$, we have $\a_{i}(r_{0},s_{0})>3^{-1}\pi$
for all $i\in\{1,2,4,5,6\}$ and $\a_{3}(r_{0},s_{0})=3^{-1}\pi$.
Therefore, since $\xi^{(A)}\cdot\a(r_{0},s_{0})=2\pi$, we have $\xi^{(A)}\cdot(1,1,1,1,1,1)<3\pi^{-1}(\xi^{(A)}\cdot\a(r_{0},s_{0}))=6.$

On the other hand, since $0<s_{0}<r_{0}<1$, we have 
\[
\a_{1}(r_{0},s_{0}),\,\a_{2}(r_{0},s_{0}),\,\a_{4}(r_{0},s_{0}),\,2\a_{5}(r_{0},s_{0}),\,2\a_{6}(r_{0},s_{0})<\pi
\]
 and $\a_{3}(r_{0},s_{0})=3^{-1}\pi$. Therefore, since $\xi^{(A)}\cdot\a(r_{0},s_{0})=2\pi$,
we obtain $12=6\pi^{-1}(\xi^{(A)}\cdot\a(r_{0},s_{0}))<\xi^{(A)}\cdot(6,6,2,6,3,3).$
\end{proof}
\begin{prop}
\label{prop:r-upper-lower-bounds}Let $(r_{0},s_{0})\in\Pi_{3}$ and
let $P$ be any compact $3$-packing with $\rad(P)=\{s_{0},r_{0},1\}$.
For any circle $B\in P$ of radius $r_{0}$ with $\xi^{(B)}\cdot(1,0,1,1,1,1)\neq0$,
we have $\xi^{(B)}\cdot(1,1,0,1,0,0)<6$ and $\xi^{(B)}\cdot(6,2,2,3,3,2)>12.$
\end{prop}

\begin{proof}
Since $0<s_{0}<r_{0}<1$, we have $\b_{1}(r_{0},s_{0}),\,\b_{4}(r_{0},s_{0})>3^{-1}\pi$,
with 
\[
\b_{3}(r_{0},s_{0}),\,\b_{5}(r_{0},s_{0}),\,\b_{6}(r_{0},s_{0})>0
\]
 and $\b_{2}(r_{0},s_{0})=3^{-1}\pi$. Therefore, since $\xi^{(B)}\cdot\b(r_{0},s_{0})=2\pi$,
we obtain $\xi^{(B)}\cdot(1,1,0,1,0,0)<3\pi^{-1}(\xi^{(B)}\cdot\b(r_{0},s_{0}))=6.$

On the other hand, since $0<s_{0}<r_{0}<1$, we have 
\[
\b_{1}(r_{0},s_{0}),\,3\b_{3}(r_{0},s_{0}),\,2\b_{4}(r_{0},s_{0}),\,2\b_{5}(r_{0},s_{0}),\,3\b_{6}(r_{0},s_{0})<\pi
\]
and $\b_{2}(r_{0},s_{0})=3^{-1}\pi$. Therefore, since $\xi^{(B)}\cdot\b(r_{0},s_{0})=2\pi$,
we obtain $\xi^{(B)}\cdot(6,2,2,3,3,2)>6\pi^{-1}(\xi^{(B)}\cdot\b(r_{0},s_{0}))=12.$
\end{proof}
\begin{prop}
\label{prop:Few-large-neighbors}Let $(r_{0},s_{0})\in\Pi_{3}$ and
let $P$ be any compact $3$-packing with $\rad(P)=\{s_{0},r_{0},1\}$.
For every circle $B\in P$ with radius $r_{0}$, we have
\[
\xi^{(B)}\cdot(2,2,0,2,1,1)\leq12,
\]
with the inequality strict if $\xi^{(B)}\cdot(2,0,0,1,1,0)>0$.
\end{prop}

\begin{proof}
Let $B$ be any circle in $P$ with radius $r_{0}$ and let $N$ be
the set of neighbors of $B$. Then 
\begin{align*}
\xi^{(B)}\cdot(2,0,0,1,1,0) & =2|\{C\in N|\radius(C)=1\}|,\\
\xi^{(B)}\cdot(0,2,0,1,0,1) & =2|\{C\in N|\radius(C)=r_{0}\}|.
\end{align*}
However, all the angles formed at the center of $B$ by connecting
the center of $B$ with the centers of circles from $\{C\in N|\radius(C)\in\{r_{0},1\}\}$
are greater or equal to $\pi/3$ and add up to $2\pi$. Therefore
$|\{C\in N|\radius(C)\in\{r_{0},1\}\}|\leq6$, and hence 
\begin{align*}
\xi^{(B)}\cdot(2,2,0,2,1,1) & =\xi^{(B)}\cdot(0,2,0,1,0,1)+\xi^{(B)}\cdot(2,0,0,1,1,0)\\
 & =2|\{C\in N|\radius(C)=r_{0}\}|+2|\{C\in N|\radius(C)=1\}|\\
 & =2|\{C\in N|\radius(C)\in\{r_{0},1\}\}|\\
 & \leq12.
\end{align*}

If $\xi^{(B)}\cdot(2,0,0,1,1,0)>0,$ then $\{C\in N|\radius(C)=1\}\neq\emptyset$.
Hence at least one of the angles formed at the center of $B$ by connecting
the center of $B$ with the centers of circles from $\{C\in N|\radius(C)\in\{r_{0},1\}\}$
is greater than $\pi/3$, with all of them still adding up to $2\pi$.
Therefore we must have $|\{C\in N|\radius(C)\in\{r_{0},1\}\}|<6$,
and the result follows. 
\end{proof}
By collecting the previous propositions into the following theorem,
we note, for every $(r_{0},s_{0})\in\Pi_{3}$ and compact $3$-packing
$P$ with $\rad(P)=\{s_{0},r_{0},1\}$, that there must necessarily
exist three circles in $P$ with respective radii $s_{0}$, $r_{0}$
and $1$, whose angle-counts satisfy the stated conditions.
\begin{thm}
\label{thm:necessary-conditions-on-r-s-tuples}Let $(r_{0},s_{0})\in\Pi_{3}$
and let $P$ be any compact $3$-packing with $\rad(P)=\{s_{0},r_{0},1\}$.
There exist circles $A$, $B$ and $C$ in $P$ with respective radii
$s_{0}$, $r_{0}$ and $1$ so that:
\begin{enumerate}
\item $\xi^{(A)}\cdot(1,0,0,1,1,0)\neq0$ or $\xi^{(B)}\cdot(1,0,0,1,1,0)\neq0$.
\item $\xi^{(A)}\cdot(1,1,0,1,1,1)\neq0$. 
\item $\xi^{(A)}\cdot(1,1,1,1,1,1)<6$ and $\xi^{(A)}\cdot(6,6,2,6,3,3)>12$.
\item $\xi^{(B)}\cdot(1,0,1,1,1,1)\neq0$.
\item $\xi^{(B)}\cdot(1,1,0,1,0,0)<6$ and $\xi^{(B)}\cdot(6,2,2,3,3,2)>12.$
\item $\xi^{(B)}\cdot(2,2,0,2,1,1)\leq12$.
\item If $\xi^{(B)}\cdot(2,0,0,1,1,0)>0,$ then $\xi^{(B)}\cdot(2,2,0,2,1,1)<12.$ 
\item $\xi^{(C)}\cdot(0,1,1,1,1,1)\neq0$ and $\xi^{(C)}\cdot(1,0,0,0,0,0)<6$.
\item For $D\in\{A,B,C\}$ we have 
\begin{align*}
\xi^{(D)}\cdot(2,0,0,1,1,0) & =0\mod2,\\
\xi^{(D)}\cdot(0,2,0,1,0,1) & =0\mod2,\\
\xi^{(D)}\cdot(0,0,2,0,1,1) & =0\mod2.
\end{align*}
\item For $D\in\{A,B,C\}$ there exists some $n\in\N$ and $\sigma\in\{s_{0},r_{0},1\}^{\{0,\ldots,n-1\}}$
with
\[
\xi^{(D)}=\sum_{j=0}^{n-1}\kappa_{\sigma(j),\sigma(j+1\mod n)}.
\]
\end{enumerate}
\end{thm}

\begin{proof}
Proposition~\ref{prop:either-s-or-r-circle-must-touch-1-circle}
yields circles $A$ and $B$ in $P$ with respective radii $s_{0}$
and $r_{0}$ so that 
\[
\xi^{(A)}\cdot(1,0,0,1,1,0)\neq0\quad\text{or}\quad\xi^{(B)}\cdot(1,0,0,1,1,0)\neq0.
\]
If $\xi^{(A)}\cdot(1,1,0,1,1,1)=0$, then we also have $\xi^{(A)}\cdot(1,0,0,1,1,0)=0$.
Since the above disjunction is true, we must then have $\xi^{(B)}\cdot(1,0,0,1,1,0)\neq0$,
which implies $\xi^{(B)}\cdot(1,0,1,1,1,1)\neq0$. Proposition~\ref{prop:not-in-center-of-hexagon}
yields some $A'\in P$ of radius $s_{0}$, so that $\xi^{(A')}\cdot(1,1,0,1,1,1)\neq0$
and $\xi^{(A')}\cdot(0,0,1,0,0,0)<6$ and we redefine $A$ as $A'$.
We may employ a similar argument to redefine $B$ if $\xi^{(B)}\cdot(1,0,1,1,1,1)=0$.
This establishes (1), (2) and (4)

By Proposition~\ref{prop:not-in-center-of-hexagon} there exists
a circle $C\in P$ satisfying $\xi^{(C)}\cdot(0,1,1,1,1,1)\neq0$
and $\xi^{(C)}\cdot(1,0,0,0,0,0)<6$, establishing (8).

The remaining assertions (3), (5), (6), (7), (9) and (10) follow immediately
from Propositions~\ref{prop:mod-2-condition}\textendash \ref{prop:Few-large-neighbors}
and the definition of the angle-counts $\xi^{(A)}$, $\xi^{(B)}$
and $\xi^{(C)}$.
\end{proof}
Motivated by the previous result, we will define a number of predicates
on $\tuples$ which will hopefully improve readability of the subsequent
sections. These predicates are named in what is hoped to be a meaningful
manner (even if some of their meanings might only become apparent
in the next section). These predicates can easily be implemented on
a computer.
\begin{defn}
\label{def:predicates}Let $\eta,\zeta\in\tuples$ and let $(r_{0},s_{0})\in\Delta_{3}$
(here we regard $r_{0}$ and $s_{0}$ purely as distinct index symbols).
We define the following predicates: 
\begin{align*}
\seq{\eta} & :=\exists n\in\N,\ \exists\sigma\in\{s_{0},r_{0},1\}^{\{0,\ldots,n-1\}},\ \eta=\sum_{j=0}^{n-1}\kappa_{\sigma(j),\sigma(j+1\mod n)},\\
\modtwo{\eta} & :=(\eta\cdot(2,0,0,1,1,0)=0\mod2)\\
 & \phantom{:=}\quad\wedge(\eta\cdot(0,2,0,1,0,1)=0\mod2)\\
 & \phantom{:=}\quad\wedge(\eta\cdot(0,0,2,0,1,1)=0\mod2),
\end{align*}
\begin{align*}
\sbounds{\eta} & :=(\eta\cdot(1,1,1,1,1,1)<6)\wedge(\eta\cdot(6,6,2,6,3,3)>12),\\
\snonhex{\eta} & :=(\eta\cdot(1,1,0,1,1,1)\neq0),\\
\snec{\eta} & :=\sbounds{\eta}\wedge\snonhex{\eta}\wedge\seq{\eta}\wedge\modtwo{\eta},
\end{align*}
\begin{align*}
\rbounds{\zeta} & :=(\zeta\cdot(1,1,0,1,0,0)<6)\wedge(\zeta\cdot(6,2,2,3,3,2)>12),\\
 & \phantom{:=}\quad\wedge(\zeta\cdot(2,2,0,2,1,1)\leq12),\\
\rnonhex{\zeta} & :=(\zeta\cdot(1,0,1,1,1,1)\neq0),\\
\rfewlargeneighbors{\zeta} & :=(\zeta\cdot(2,0,0,1,1,0)>0)\Rightarrow(\zeta\cdot(2,2,0,2,1,1)<12),\\
\rnec{\zeta} & :=\rbounds{\zeta}\wedge\rnonhex{\zeta}\wedge\seq{\zeta}\wedge\modtwo{\zeta},\\
 & \phantom{:=}\quad\wedge\rfewlargeneighbors{\zeta}\\
\rverticalcont{\zeta} & :=\zeta\cdot(0,0,1,0,1,1)=0,
\end{align*}
\begin{align*}
\srdisjunct{\eta,\zeta} & :=(\eta\cdot(1,0,0,1,1,0)\neq0)\vee(\zeta\cdot(1,0,0,1,1,0)\neq0),
\end{align*}
\begin{align*}
\ononhex{\eta} & :=(\eta\cdot(0,1,1,1,1,1)\neq0)\wedge(\eta\cdot(1,0,0,0,0,0)<6),\\
\onec{\eta} & :=\ononhex{\eta}\wedge\seq{\eta}\wedge\modtwo{\eta}.
\end{align*}
\end{defn}

Now, a straightforward brute-force search by computer can establish
that the set $\set{\eta\in\tuples}{\snec{\eta}}$ is finite and has
exactly 55 elements, see Proposition~\ref{prop:55-s-angle-counts}
below. 

We note that the 3rd coordinate of elements from $\set{\xi\in\tuples}{\rnec{\xi}}$
is not bounded above, and hence this set may be infinite. The next
two sections will address this issue.
\begin{prop}
\label{prop:55-s-angle-counts}The set $\set{\eta\in\tuples}{\snec{\eta}}$
has exactly 55 elements, and its members are listed in Table~\ref{tab:s-nec-elements}.
\end{prop}

\begin{table}[h]
\[
\begin{array}{ccccc}
(0,0,0,1,1,3) & (0,0,1,2,2,0) & (0,1,2,0,0,2) & (1,0,0,0,4,0) & (1,2,0,2,0,0)\\
(0,0,0,1,3,1) & (0,0,2,1,1,1) & (0,2,0,0,0,2) & (1,0,0,1,1,1) & (2,0,0,0,2,0)\\
(0,0,0,2,0,2) & (0,1,0,0,0,4) & (0,2,0,1,1,1) & (1,0,0,2,0,0) & (2,0,0,1,1,1)\\
(0,0,0,2,2,0) & (0,1,0,0,2,2) & (0,2,0,2,0,0) & (1,0,0,2,0,2) & (2,0,0,2,0,0)\\
(0,0,0,3,1,1) & (0,1,0,1,1,1) & (0,2,1,0,0,2) & (1,0,0,2,2,0) & (2,0,1,0,2,0)\\
(0,0,0,4,0,0) & (0,1,0,2,0,0) & (0,3,0,0,0,0) & (1,0,0,4,0,0) & (2,1,0,2,0,0)\\
(0,0,1,0,0,4) & (0,1,0,2,0,2) & (0,3,0,0,0,2) & (1,0,1,0,2,0) & (3,0,0,0,0,0)\\
(0,0,1,0,2,2) & (0,1,0,2,2,0) & (0,3,0,2,0,0) & (1,0,1,1,1,1) & (3,0,0,0,2,0)\\
(0,0,1,0,4,0) & (0,1,0,4,0,0) & (0,4,0,0,0,0) & (1,0,2,0,2,0) & (3,0,0,2,0,0)\\
(0,0,1,1,1,1) & (0,1,1,0,0,2) & (0,5,0,0,0,0) & (1,1,0,1,1,1) & (4,0,0,0,0,0)\\
(0,0,1,2,0,2) & (0,1,1,1,1,1) & (1,0,0,0,2,2) & (1,1,0,2,0,0) & (5,0,0,0,0,0).
\end{array}
\]
\caption{\label{tab:s-nec-elements}The 55 members of the set $\protect\set{\eta\in\protect\tuples}{\protect\snec{\eta}}$. }
\end{table}

\section{Contour analysis\label{sec:Contour-analysis}}

Theorem~\ref{thm:necessary-conditions-on-r-s-tuples} provides no
upper bound on the 3rd coordinate of angle-counts for midsize circles
in a compact $3$-packing. In this section, for arbitrary elements
$\eta\in\set{\xi\in\tuples}{\snec{\xi}}$ and $\zeta\in\set{\xi\in\tuples}{\rnec{\xi}}$,
we will analyze the properties of the $2\pi$-contours of the functions
\begin{align*}
 & \Delta_{3}\ni(r,s)\mapsto\eta\cdot\a(r,s),\\
 & \Delta_{3}\ni(r,s)\mapsto\zeta\cdot\b(r,s).
\end{align*}
The main goal in this section is establishing an upper bound for the
3rd coordinate of angle-counts for midsize circles in a compact $3$-packing,
through this contour analysis.

We begin with an analysis of the $2\pi$-contours of $\Delta_{3}\ni(r,s)\mapsto\eta\cdot\a(r,s)$
in Proposition~\ref{prop:s-contour-analysis}. A crucial part of
Proposition~\ref{prop:s-contour-analysis} is (\ref{enu:s-contour-global-bound}),
which explicitly describes a region in $\Delta_{3}$ containing the
$2\pi$-contours of $\Delta_{3}\ni(r,s)\mapsto\eta\cdot\a(r,s)$ for
all $\eta\in\set{\xi\in\tuples}{\snec{\xi}}$. 

Subsequently, in Proposition~\ref{prop:r-contour-analysis}(\ref{enu:r-contour-global-bound}),
we prove that if the 3rd coordinate of $\zeta\in\set{\xi\in\tuples}{\rnec{\xi}}$
is too large, then the $2\pi$-contour of $\Delta_{3}\ni(r,s)\mapsto\zeta\cdot\b(r,s)$
lies in a region that is disjoint from the region containing the $2\pi$-contours
of $\Delta_{3}\ni(r,s)\mapsto\eta\cdot\a(r,s)$ for all $\eta\in\set{\xi\in\tuples}{\snec{\xi}}$.
Therefore, these contours cannot intersect, while such an intersection
is a necessary condition for all angle-counts for circles in a compact
$3$-packing, as described in Section~\ref{sec:Preliminaries}. This
allows us to establish a bound on the 3rd coordinate of angle-counts
for midsize circles in a compact $3$-packing and lays the groundwork
for showing that $\Pi_{3}$ is finite in the next section.
\begin{prop}
\label{prop:s-contour-analysis}Let $\eta\in\set{\xi\in\tuples}{\snec{\xi}}$.
For any $m\in(0,1)$, we define the functions $f_{\eta}:\Delta_{3}\to\R$
and $g_{\eta,m}:(0,1)\to\R$ by 
\begin{align*}
f_{\eta}(r,s) & :=\eta\cdot\a(r,s)\quad(r,s)\in\Delta_{3}\\
g_{\eta,m}(r) & :=\eta\cdot\a(r,mr)\quad r\in(0,1).
\end{align*}
Then:
\begin{enumerate}
\item \itemsep7pt \label{enu:s-contour-non-empty-contour}There exists
some $(r_{0},s_{0})\in\Delta_{3}$ with $f_{\eta}(r_{0},s_{0})=2\pi.$
\item \label{enu:s-contour-g-monotone}For any $m\in(0,1)$ the function
$g_{\eta,m}$ is monotone decreasing and, if $\eta\cdot(1,0,0,1,1,0)\neq0$,
then $g_{\eta,m}$ is strictly decreasing. If $\eta\cdot(1,0,0,1,1,0)=0$
then $g_{\eta,m}$ is a constant function.
\item \label{enu:s-contour-s-partial-der}We have $(\partial_{2}f_{\eta})(r,s)<0$
for all $(r,s)\in\Delta_{3}$.
\item \label{enu:s-contour-implicit-function}There exists some $a\in[0,1)$
and a differentiable function $\phi:(a,1)\to(0,1)$ so that $f_{\eta}(r,\phi(r))=2\pi$
for all $r\in(a,1).$ We may choose $a\in[0,1)$ so that the graph
of $\phi$ equals the whole contour $\set{(r,s)\in\Delta_{3}}{f_{\eta}(r,s)=2\pi}.$
\item \label{enu:s-contour-intercept-with-phi}Let $(r_{0},s_{0})\in\Delta_{3}$
be any point satisfying $f_{\zeta}(r_{0},s_{0})=2\pi$ and let $\phi:(a,1)\to(0,1)$
be as yielded by (\ref{enu:s-contour-implicit-function}). With $m_{0}:=s_{0}/r_{0}$,
we have \smallskip
\begin{enumerate}
\item \itemsep7pt If $\eta\cdot(1,0,0,1,1,0)\neq0$, then $\phi(r)>m_{0}r$
for $r\in(a,r_{0})$ and $\phi(r)<m_{0}r$ for $r\in(r_{0},1)$. 
\item If $\eta\cdot(1,0,0,1,1,0)=0$, then $\phi(r)=m_{0}r$ for $r\in(0,1)$. 
\end{enumerate}
\item \label{enu:s-contour-global-bound} We have $f_{\eta}(r,s)>2\pi$
for all $(r,s)\in\Delta_{3}$ satisfying $s\leq10^{-1}r$.
\end{enumerate}
\end{prop}

\begin{figure}[t]
\includegraphics[width=1\textwidth]{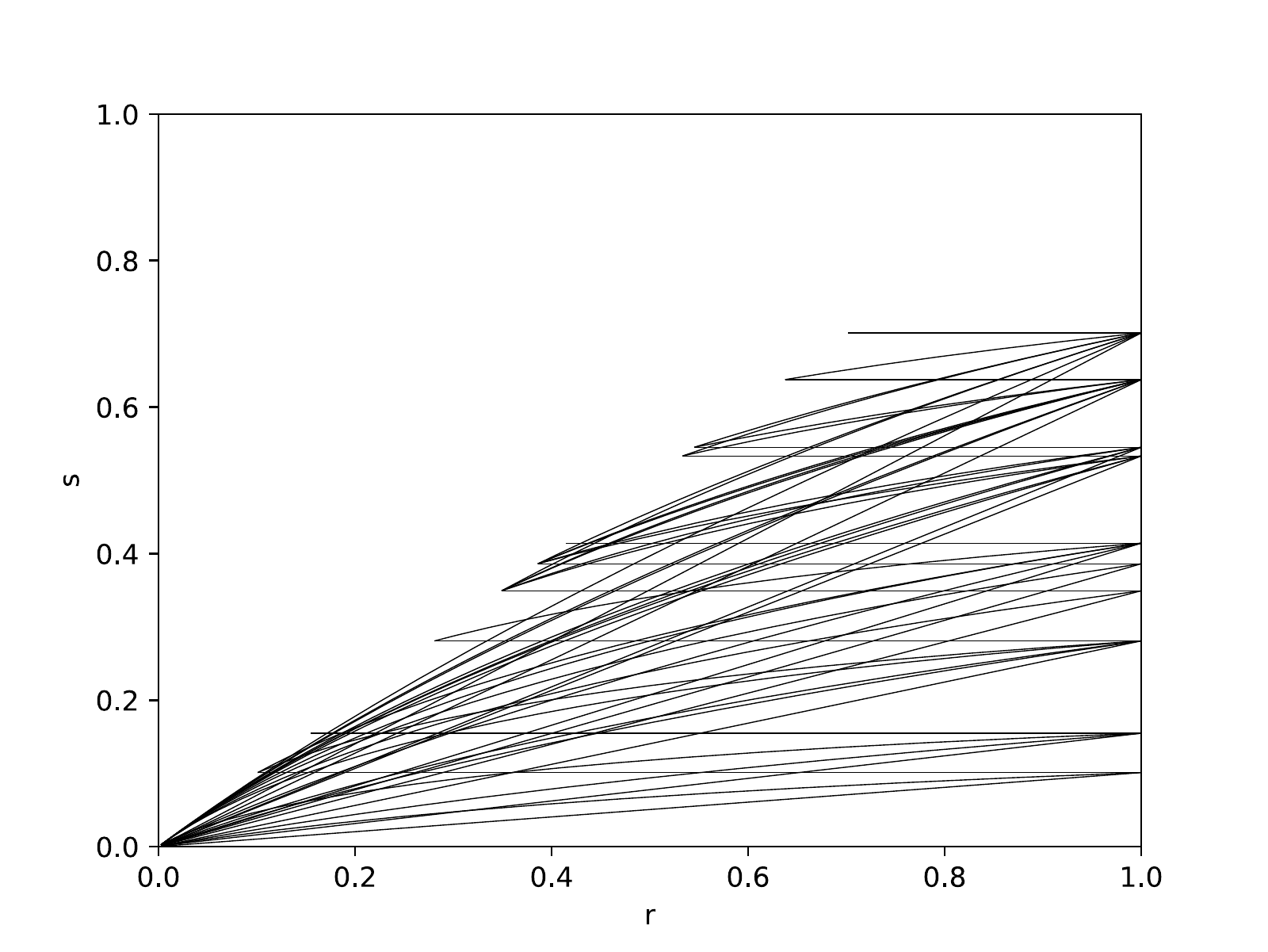}\caption{\label{fig:s-contours}The contours $\protect\set{(r,s)\in\Delta_{3}}{\eta\cdot\protect\a(r,s)=2\pi}$
for the 55 elements $\eta\in\protect\set{\xi\in\protect\tuples}{\protect\snec{\xi}}$. }
\end{figure}

\begin{proof}
We prove (\ref{enu:s-contour-non-empty-contour}). On $\set{(r,s)}{0\leq s\leq r\leq1}$,
each of the functions $(r,s)\mapsto\a_{i}(r,s)$ for $i\in\{1,\ldots,6\}$
attains its minimum and maximum respectively at $(1,1)$ and $(1,0)$.
Since $\snec{\eta}$ is true, we have $\eta\cdot(1,1,1,1,1,1)<6$
and $\eta\cdot(6,6,2,6,3,3)>12$, and therefore there exists some
$(r_{1},s_{1})\in\Delta_{3}$ (close to $(1,1)$) with $f_{\eta}(r_{1},s_{1})<2\pi$
and there exists some $(r_{2},s_{2})\in\Delta_{3}$ (close to $(1,0)$)
with $f_{\eta}(r_{2},s_{2})>2\pi$. By the Intermediate Value Theorem,
there exists some $(r_{0},s_{0})\in\Delta_{3}$ with $f_{\eta}(r_{0},s_{0})=2\pi$.

The assertion (\ref{enu:s-contour-g-monotone}) can be verified by
using a computer algebra system. Explicitly, for $m\in(0,1)$, where,
for $r\in(0,1)$,
\begin{align*}
U(r) & :=\frac{2m}{(mr+1)\sqrt{mr(mr+2)}},\\
V(r) & :=\frac{m}{(mr+1)\sqrt{m(mr+r+1)}},\\
W(r) & :=\frac{m}{(mr+1)\sqrt{2mr+1}},
\end{align*}
the function $g_{\eta,m}$ has derivative
\[
g_{\eta,m}'(r)=-\eta\cdot(U(r),0,0,V(r),W(r),0)\quad(r\in(0,1)).
\]
Furthermore, this derivative is seen to be everywhere non-positive,
and everywhere strictly negative if $\eta\cdot(1,0,0,1,1,0)\neq0$
and zero when $\eta\cdot(1,0,0,1,1,0)=0$, establishing (\ref{enu:s-contour-g-monotone}). 

We prove (\ref{enu:s-contour-s-partial-der}). For $i\in\{1,2,3,4,5,6\}$,
the functions $\Delta_{3}\ni(r,s)\mapsto\a_{i}(r,s)$ all have non-positive
(strictly negative for $i\in\{1,2,4,5,6\}$) partial derivatives with
respect to the second parameter everywhere on $\Delta_{3}$. Since
$\snonhex{\eta}$ is true, we have that$(\partial_{2}f_{\eta})(r,s)<0$
for all $(r,s)\in\Delta_{3}$, establishing (\ref{enu:s-contour-s-partial-der}).

We prove (\ref{enu:s-contour-implicit-function}). We define $G:=\set{r\in(0,1)}{\exists s\in(0,r),\ f_{\eta}(r,s)=2\pi},$
which is non-empty by (\ref{enu:s-contour-non-empty-contour}). By
(\ref{enu:s-contour-s-partial-der}), for every $r\in G$, there exists
a unique $\phi(r)\in(0,r)$ satisfying $f_{\eta}(r,\phi(r))=2\pi$.
It is clear that the graph of $\phi$ equals the contour $\set{(r,s)\in\Delta_{3}}{f_{\eta}(r,s)=2\pi}.$
By (\ref{enu:s-contour-s-partial-der}) and the Implicit Function
Theorem, the set $G$ is open and $\phi:G\to(0,1)$ is differentiable.
Since $(\partial_{1}f_{\eta})(r,s)\geq0$ for all $(r,s)\in\Delta_{3}$,
the set $G$ is connected and hence is an open interval $(a,b)$.
Furthermore, it can be verified (by computer) that we may always choose
$b=1$ (See Figure~\ref{fig:s-contours}).

We prove (\ref{enu:s-contour-intercept-with-phi}). If $\eta\cdot(1,0,0,1,1,0)\neq0$,
by (\ref{enu:s-contour-g-monotone}), the function $g_{\eta,m_{0}}$
is strictly decreasing. Therefore $f_{\eta}(r,rm_{0})=g_{\eta,m_{0}}(r)>2\pi$
for all $r\in(a,r_{0})$ and $f_{\eta}(r,rm_{0})=g_{\eta,m_{0}}(r)<2\pi$
for all $r\in(r_{0},1)$. But, by (\ref{enu:s-contour-s-partial-der}),
we have $(\partial_{2}f_{\eta})(r,s)<0$ for all $(r,s)\in\Delta_{3}$,
so that we must have $\phi(r)>m_{0}r$ for $r\in(a,r_{0})$, and $\phi(r)<m_{0}r$
for $r\in(r_{0},1)$. If $\eta\cdot(1,0,0,1,1,0)=0$, by (\ref{enu:s-contour-g-monotone}),
the function $g_{\eta,m_{0}}$ constant, and since $g_{\eta,m_{0}}(r_{0})=2\pi$,
we have that $\phi$ must equal the function $(0,1)\ni r\mapsto m_{0}r$. 

We prove (\ref{enu:s-contour-global-bound}). Figure~\ref{fig:s-contours}
may be a helpful visual aid. It can be verified (by computer) that
$\lim_{r\to1}f_{\eta}\parenth{r,10^{-1}r}>2\pi$. Then, by (\ref{enu:s-contour-g-monotone}),
we obtain $f_{\eta}\parenth{r,10^{-1}r}>2\pi$ for all $r\in(0,1)$.
But, by (\ref{enu:s-contour-s-partial-der}), we have $(\partial_{2}f_{\eta})(r,s)<0$
for all $(r,s)\in\Delta_{3}$, so that $f_{\eta}(r,s)>2\pi$ for all
$(r,s)\in\Delta_{3}$ satisfying $s\leq10^{-1}r$, establishing (\ref{enu:s-contour-global-bound}). 
\end{proof}
\begin{prop}
\label{prop:r-contour-analysis}Let $\zeta\in\set{\xi\in\tuples}{\rnec{\xi}}$
be arbitrary. For any $m\in(0,1)$, we define the functions $f_{\zeta}:F\to\R$
and $g_{\zeta,m}:(0,1)\to\R$ by 
\begin{align*}
f_{\zeta}(r,s) & :=\zeta\cdot\b(r,s)\quad(r,s)\in\Delta_{3}\\
g_{\zeta,m}(r) & :=\zeta\cdot\b(r,mr)\quad r\in(0,1).
\end{align*}
Then:
\begin{enumerate}
\item \itemsep7pt \label{enu:r-contour-non-empty-contour}There exists
some $(r_{0},s_{0})\in\Delta_{3}$ with $f_{\zeta}(r_{0},s_{0})=2\pi.$
\item \label{enu:r-contour-g-monotone}For any $m\in(0,1)$, the function
$g_{\zeta,m}$ is monotone decreasing and, if $\zeta\cdot(1,0,0,1,1,0)\neq0$,
then $g_{\zeta,m}$ is strictly decreasing. If $\zeta\cdot(1,0,0,1,1,0)=0$,
then $g_{\zeta,m}$ is a constant function.
\item \label{enu:r-contour-s-partial-der}If $\rverticalcont{\zeta}$ is
false, then $(\partial_{2}f_{\zeta})(r,s)>0$ for all $(r,s)\in\Delta_{3}$.
\item \label{enu:r-contour-vertical-contours}If $\rverticalcont{\zeta}$
is true, then there exists $r_{0}\in(0,1)$ for which $f_{\zeta}(r_{0},s)=2\pi$
for all $s\in(0,r_{0})$.
\item \label{enu:r-contour-implicit-function}If $\rverticalcont{\zeta}$
is false, then there exists some differentiable function $\psi:(c,d)\to(0,1)$
so that $f_{\zeta}(r,\psi(r))=2\pi$ for all $r\in(c,d).$ We may
choose the interval $(c,d)$ so that the graph of $\psi$ equals the
entire contour $\set{(r,s)\in\Delta_{3}}{f_{\zeta}(r,s)=2\pi}.$
\item \label{enu:r-contour-intercept-with-phi}Let $\rverticalcont{\zeta}$
be false and $(r_{0},s_{0})\in\Delta_{3}$ be any point satisfying
$f_{\zeta}(r_{0},s_{0})=2\pi$. With $\psi:(c,d)\to\R$ as yielded
by (\ref{enu:r-contour-implicit-function}) and $m_{0}:=s_{0}/r_{0}$,
\smallskip
\begin{enumerate}
\item \itemsep7pt If $\zeta\cdot(1,0,0,1,1,0)\neq0$, then $\psi(r)<m_{0}r$
for $r\in(c,r_{0})$ and $\psi(r)>m_{0}r$ for $r\in(r_{0},d)$. 
\item If $\zeta\cdot(1,0,0,1,1,0)=0$, then $\psi(r)=m_{0}r$ for all $r\in(0,1)$.
\end{enumerate}
\item \label{enu:r-contour-global-bound}If $\zeta_{3}\geq35$, then $f_{\zeta}(r,s)>2\pi$
for all $(r,s)\in\Delta_{3}$ satisfying $s\geq10^{-1}r.$
\end{enumerate}
\end{prop}

\begin{proof}
We prove (\ref{enu:r-contour-non-empty-contour}). On $\set{(r,s)}{0\leq s\leq r\leq1}$,
each of the functions $(r,s)\mapsto\b_{i}(r,s)$ for $i\in\{1,\ldots,6\}$
attains its minimum and maximum respectively at $(1,0)$ and $(0,0)$
(or approached near $(0,0)$, if the function is not defined at $(0,0)$).
Since $\rnec{\zeta}$ is true, we have $\zeta\cdot(1,1,0,1,0,0)<6$
and $\zeta\cdot(6,2,2,3,3,2)>12$, and hence there exists some $(r_{1},s_{1})\in\Delta_{3}$
(close to $(1,0)$) with $f_{\zeta}(r_{1},s_{1})<2\pi$ and there
exists some $(r_{2},s_{2})\in\Delta_{3}$ (close to $(0,0)$) with
$f_{\zeta}(r_{2},s_{2})>2\pi$. By the Intermediate Value Theorem,
there exists some $(r_{0},s_{0})\in\Delta_{3}$ with $f_{\zeta}(r_{0},s_{0})=2\pi$.

The assertion (\ref{enu:r-contour-g-monotone}) can be verified with
a computer algebra system. Explicitly, for $r\in(0,1)$ and $m\in(0,1)$,
defining 
\begin{align*}
U(r) & :=\frac{2}{(r+1)\sqrt{r(r+2)}},\\
V(r) & :=\frac{1}{(r+1)\sqrt{2r+1}},\\
W(r) & :=\frac{m}{(r+1)\sqrt{m(mr+r+1)}},
\end{align*}
the function $g_{\zeta,m}$ has derivative
\[
g_{\zeta,m}'(r)=-\zeta\cdot(U(r),0,0,V(r),W(r),0)\quad(r\in(0,1)),
\]
which is easily seen to be non-positive, and strictly negative if
$\zeta\cdot(1,0,0,1,1,0)\neq0$, and zero when $\zeta\cdot(1,0,0,1,1,0)=0$. 

We prove (\ref{enu:r-contour-s-partial-der}). For $i\in\{1,2,3,4,5,6\}$,
the functions $\Delta_{3}\ni(r,s)\mapsto\b_{i}(r,s)$ all have non-negative
(strictly positive for $i\in\{3,5,6\}$) partial derivatives with
respect to the second parameter everywhere on $\Delta_{3}$.  Therefore,
if $\zeta\cdot(0,0,1,0,1,1)\neq0$, then $(\partial_{2}f_{\zeta})(r,s)>0$
for all $(r,s)\in\Delta_{3}$, establishing (\ref{enu:r-contour-s-partial-der}).

The assertion (\ref{enu:r-contour-vertical-contours}) follows from
(\ref{enu:r-contour-non-empty-contour}) when we notice that, for
$i\in\{1,2,4\}$, the functions $\Delta_{3}\ni(r,s)\mapsto\b_{i}(r,s)$
are all independent of the second parameter $s$.

We prove (\ref{enu:r-contour-implicit-function}). Define $G:=\set{r\in(0,1)}{\exists s\in(0,r),\ f_{\zeta}(r,s)=2\pi},$
which is non-empty by (\ref{enu:r-contour-non-empty-contour}). By
(\ref{enu:r-contour-s-partial-der}), for every $r\in G$, there exists
a unique $\psi(r)\in(0,r)$ such that $f_{\zeta}(r,\psi(r))=2\pi$.
It is clear that graph of the function $\psi:G\to(0,1)$ equals the
contour $\set{(r,s)\in\Delta_{3}}{f_{\zeta}(r,s)=2\pi}.$ By (\ref{enu:r-contour-s-partial-der}),
and the Implicit Function Theorem, $G$ is open and the function $\psi:G\to(0,1)$
is differentiable. Since $(\partial_{1}f_{\zeta})(r,s)\leq0$ for
all $(r,s)\in\Delta_{3}$, the set $G$ is connected, and hence must
be some open interval $(c,d)$.

We prove (\ref{enu:r-contour-intercept-with-phi}). If $\eta\cdot(1,0,0,1,1,0)\neq0$,
by (\ref{enu:r-contour-g-monotone}), the function $g_{\zeta,m_{0}}$
is strictly decreasing, and hence $f_{\zeta}(r,rm_{0})=g_{\zeta,m_{0}}(r)>2\pi$
for $r\in(c,r_{0})$ and $f_{\zeta}(r,rm_{0})=g_{\zeta,m_{0}}(r)<2\pi$
for $r\in(r_{0},d)$. By (\ref{enu:r-contour-s-partial-der}), we
have $(\partial_{2}f_{\zeta})(r,s)>0$ for all $(r,s)\in\Delta_{3}$,
and therefore $\psi(r)<m_{0}r$ for $r\in(c,r_{0})$ and $\psi(r)>m_{0}r$
for $r\in(r_{0},d)$. On the other hand, if $\eta\cdot(1,0,0,1,1,0)=0,$
the function $g_{\zeta,m_{0}}$ is constant, and since $g_{\zeta,m_{0}}(r_{0})=2\pi$,
we have that $\psi$ equals $(0,1)\ni r\mapsto m_{0}r$.

We prove (\ref{enu:r-contour-global-bound}). We assume that $\zeta_{3}\geq35$.
A straightforward computation shows that, for all $r\in(0,1)$, we
have $35\b_{3}(r,10^{-1}r)>2\pi.$ Therefore $f_{\zeta}(r,10^{-1}r)=\zeta\cdot\b(r,10^{-1}r)>2\pi$
for all $r\in(0,1)$. Since $\zeta_{3}\geq35$, we have $\zeta\cdot(0,0,1,0,1,1)\neq0$,
so that, by (\ref{enu:r-contour-s-partial-der}), we have $(\partial_{2}f_{\zeta})(r,s)>0$
for all $(r,s)\in\Delta_{3}$. Hence $f_{\zeta}(r,s)>2\pi$ for all
$(r,s)\in\Delta_{3}$ with $s\geq10^{-1}r.$
\end{proof}

\section{\label{sec:Pi_3 is finite}The set $\Pi_{3}$ is finite}

We are now in a position to prove one of our main results, Theorem~\ref{thm:set-radii-admitting-3packings-are-finite},
in this section. We begin by defining the following predicate:
\begin{defn}
For $\zeta\in\tuples$ we define the predicate
\begin{align*}
\rboundsextra{\zeta} & :=(\zeta\cdot(0,0,1,0,0,0)<35).
\end{align*}
\end{defn}

We define the set 
\[
K:=\set{(\eta,\zeta)\in\tuples^{2}}{\begin{array}{l}
\phantom{\wedge}\snec{\eta}\\
\wedge\rnec{\zeta}\\
\wedge\rboundsextra{\zeta}\\
\wedge\srdisjunct{\eta,\zeta}
\end{array}}.
\]

We will argue in this section that $|K|<\infty$ and that $|\Pi_{3}|\leq|K|$. 

A straightforward computer search will establish the cardinality of
$K$. All elements of $K$ are provided in the attached dataset.
\begin{prop}
\label{prop:tuples-satisfying-necessary-conditions-are-finite}The
set $K$ is finite and has exactly $248395$ elements.
\end{prop}

The next proposition shows that every element of $K$ determines at
most one point of $\Delta_{3}$. 

{}
\begin{prop}
\label{prop:unique-intercepts}For any pair $(\eta,\zeta)\in K,$
there exists at most one (perhaps no) point $(r_{0},s_{0})\in\Delta_{3}$
for which $\eta\cdot\a(r_{0},s_{0})=2\pi$ and $\zeta\cdot\b(r_{0},s_{0})=2\pi$. 
\end{prop}

\begin{proof}
Let $(\eta,\zeta)\in K$ be arbitrary. If there exists no point $(r,s)\in\Delta_{3}$
for which $\eta\cdot\a(r,s)=2\pi$ and $\zeta\cdot\b(r,s)=2\pi$,
then we are done. 

Let $(r_{0},s_{0})\in\Delta_{3}$ be such that $\eta\cdot\a(r_{0},s_{0})=2\pi$
and $\zeta\cdot\b(r_{0},s_{0})=2\pi$. We claim that there exists
no other point in $\Delta_{3}$ for which this is true. 

By Proposition~\ref{prop:s-contour-analysis}(\ref{enu:s-contour-implicit-function}),
there exists some $\phi:(a,1)\to\R$ so that $\phi(r_{0})=s_{0}$
and $\eta\cdot\a(r,\phi(r))=2\pi$ for all $r\in(a,1)$. 

We now distinguish between the two cases where $\rverticalcont{\zeta}$
is true and $\rverticalcont{\zeta}$ is false.

If it is the case that $\rverticalcont{\zeta}$ is true, then, by
Proposition~\ref{prop:r-contour-analysis}(\ref{enu:r-contour-vertical-contours}),
we have $\set{(r,s)\in\Delta_{3}}{\zeta\cdot\b(r,s)=2\pi}=\set{(r_{0},s)\in\Delta_{3}}{s\in(0,r_{0})}$,
and hence the pair $(r_{0},s_{0})$ is the only point in $\Delta_{3}$
for which $\eta\cdot\a(r_{0},s_{0})=2\pi$ and $\zeta\cdot\b(r_{0},s_{0})=2\pi$.

On the other hand, if $\rverticalcont{\zeta}$ is false, then, by
Proposition~\ref{prop:r-contour-analysis}(\ref{enu:r-contour-implicit-function}),
there exists some function $\psi:(c,d)\to(0,1)$ satisfying $\psi(r_{0})=s_{0}$
and $\zeta\cdot\b(r,\psi(r))=2\pi$ for all $r\in(c,d)$. Since $\srdisjunct{\eta,\zeta}$
is true, by Propositions~\ref{prop:s-contour-analysis}(\ref{enu:s-contour-g-monotone})
and~\ref{prop:r-contour-analysis}(\ref{enu:r-contour-g-monotone})
we cannot have that both functions $\phi$ and $\psi$ are equal to
the function $(0,1)\ni r\mapsto mr$ for any $m\in(0,1)$. Then, by
Propositions~\ref{prop:s-contour-analysis}(\ref{enu:s-contour-intercept-with-phi})
and~\ref{prop:r-contour-analysis}(\ref{enu:r-contour-intercept-with-phi}),
we have that $(r_{0},s_{0})$ is the only point in $\Delta_{3}$ for
which $\eta\cdot\a(r_{0},s_{0})=2\pi$ and $\zeta\cdot\b(r_{0},s_{0})=2\pi$.
\end{proof}
Now, for any $(r_{0},s_{0})\in\Pi_{3}$ and compact $3$-packing $P$
with $\rad(P)=\{s_{0},r_{0},1\}$, in the following result we will
prove that there must exist circles $A$ and $B$ of respective radii
$s_{0}$ and $r_{0}$, so that $(\xi^{(A)},\xi^{(B)})\in K$.
\begin{prop}
\label{prop:final-necessary-condition-for-3-packing}Let $(r_{0},s_{0})\in\Pi_{3}$
and let $P$ be any compact $3$-packing with $\rad(P)=\{s_{0},r_{0},1\}$.
There exists circles $A,B\in P$ of respective radii $s_{0}$ and
$r_{0}$ so that the following is true:
\begin{align*}
 & \snec{\xi^{(A)}}\wedge\rnec{\xi^{(B)}}\\
 & \quad\quad\wedge\srdisjunct{\xi^{(A)},\xi^{(B)}}\wedge\rboundsextra{\xi^{(B)}}\\
 & \quad\quad\quad\quad\wedge\ \ \xi^{(A)}\cdot\a(r_{0},s_{0})=2\pi\ \ \wedge\ \ \xi^{(B)}\cdot\b(r_{0},s_{0})=2\pi.
\end{align*}
In particular, we have  $(\xi^{(A)},\xi^{(B)})\in K.$
\end{prop}

\begin{proof}
By Theorem~\ref{thm:necessary-conditions-on-r-s-tuples} there exist
circles $A,B\in P$ so that 
\[
\snec{\xi^{(A)}}\wedge\rnec{\xi^{(B)}}\wedge\srdisjunct{\xi^{(A)},\xi^{(B)}}
\]
is true. By definition of the angle-counts $\xi^{(A)}$ and $\xi^{(B)}$,
we have that $\xi^{(A)}\cdot\a(r_{0},s_{0})=2\pi$ and $\xi^{(B)}\cdot\b(r_{0},s_{0})=2\pi$. 

Suppose that $\rboundsextra{\xi^{(B)}}$ is false. Then, since $\xi^{(B)}\cdot\b(r_{0},s_{0})=2\pi,$
by Proposition~\ref{prop:r-contour-analysis}(\ref{enu:r-contour-global-bound}),
we have that $s_{0}<10^{-1}r_{0}$. However, since $\xi^{(A)}\cdot\a(r_{0},s_{0})=2\pi$,
Proposition~\ref{prop:s-contour-analysis}(\ref{enu:s-contour-global-bound})
yields the contradictory inequality $s_{0}>10^{-1}r_{0}$. Therefore
$\rboundsextra{\xi^{(B)}}$ is true.
\end{proof}
Finally we are able to prove one of our main results:
\begin{thm}
\label{thm:set-radii-admitting-3packings-are-finite}The set $\Pi_{3}$
is finite and $|\Pi_{3}|\leq|K|=248395.$
\end{thm}

\begin{proof}
We define 
\[
L:=\set{(r,s)\in\Delta_{3}}{\begin{array}{c}
\exists(\eta,\zeta)\in K,\\
\eta\cdot\a(r,s)=2\pi,\\
\zeta\cdot\b(r,s)=2\pi.
\end{array}}.
\]
By Proposition~\ref{prop:tuples-satisfying-necessary-conditions-are-finite},
the set $K$ has $248395$ elements and for each $(\eta,\zeta)\in K$,
by Proposition~\ref{prop:unique-intercepts}, there exists at most
one point $(r_{0},s_{0})\in\Delta_{3}$ for which $\eta\cdot\a(r_{0},s_{0})=2\pi$
and $\zeta\cdot\b(r_{0},s_{0})=2\pi$. Therefore, we have $|L|\leq|K|.$ 

We claim that $\Pi_{3}\subseteq L$. Let $(r_{0},s_{0})\in\Pi_{3}$
and $P$ be a compact $3$-packing with $\rad(P)=\{s_{0},r_{0},1\}$.
By Proposition~\ref{prop:final-necessary-condition-for-3-packing},
there exist circles $A$ and $B$ in $P$ so that $(\xi^{(A)},\xi^{(B)})\in K$
and $\xi^{(A)}\cdot\a(r_{0},s_{0})=2\pi$ and $\xi^{(B)}\cdot\b(r_{0},s_{0})=2\pi$.
Therefore $(r_{0},s_{0})\in L$ and hence $\Pi_{3}\subseteq L$. We
conclude that $|\Pi_{3}|\leq|L|\leq|K|\leq248395.$
\end{proof}

\section{Necessary and sufficient conditions for contour intercepts\label{sec:Necessary-and-sufficient-conditions-for-contour-intercepts}}

With $K$ as defined in Section~\ref{sec:Pi_3 is finite}, in the
current section we will provide necessary and sufficient conditions
on elements $(\eta,\xi)\in K$ for there to exist some $(r_{0},s_{0})\in\Delta_{3}$
satisfying $\eta\cdot\a(r_{0},s_{0})=2\pi$ and $\zeta\cdot\a(r_{0},s_{0})=2\pi$.
These necessary and sufficient conditions allow for computing a sharper
bound on $|\Pi_{3}|$ in the next section.
\begin{prop}
\label{prop:intercept-necessary-and-sufficient-conditions}Let $(\eta,\zeta)\in K$
with $\phi:(a,1)\to(0,1)$ and $\psi:(c,d)\to(0,1)$ as yielded by
applying Propositions~\ref{prop:s-contour-analysis}(\ref{enu:s-contour-implicit-function})
and~\ref{prop:r-contour-analysis}(\ref{enu:r-contour-implicit-function})
to $\eta$ and $\zeta$ respectively. The statements (1) and (2) are
equivalent:
\begin{enumerate}
\item There exists a unique point $(r_{0},s_{0})\in\Delta_{3}$ for which
$\eta\cdot\a(r_{0},s_{0})=2\pi$ and $\zeta\cdot\b(r_{0},s_{0})=2\pi$.\bigskip
\item Either $\rverticalcont{\zeta}$ is true and there exists some $r_{0}\in(a,1)$
so that $\eta\cdot\a(r_{0},\phi(r_{0}))=2\pi$ and $\zeta\cdot\b(r_{0},\phi(r_{0}))=2\pi$;
\textbf{or} $\rverticalcont{\zeta}$ is false, and all of the following
hold:\medskip
\begin{enumerate}
\item $a<d$.\medskip
\item If $a=c=0$, then $\lim_{r\downarrow0}\frac{\psi(r)}{r}<\lim_{r\downarrow0}\frac{\phi(r)}{r}.$\medskip
\item If $d=1$, then $\lim_{r\uparrow1}\phi(r)<\lim_{r\uparrow1}\psi(r).$\medskip
\end{enumerate}
\end{enumerate}
\end{prop}

\begin{proof}
We prove that (1) implies (2). Let $(r_{0},s_{0})\in\Delta_{3}$ be
the unique point for which $\eta\cdot\a(r_{0},s_{0})=2\pi$ and $\zeta\cdot\b(r_{0},s_{0})=2\pi$. 

If $\rverticalcont{\zeta}$ is true then, since $s_{0}=\phi(r_{0})$,
we immediately have that $\eta\cdot\a(r_{0},\phi(r_{0}))=2\pi$ and
$\zeta\cdot\b(r_{0},\phi(r_{0}))=2\pi$.

On the other hand, if $\rverticalcont{\zeta}$ is false, with $m_{0}:=s_{0}/r_{0}$
we immediately note that, since $\srdisjunct{\eta,\zeta}$ is true,
by Propositions \ref{prop:s-contour-analysis}(\ref{enu:r-contour-intercept-with-phi})
and \ref{prop:r-contour-analysis}(\ref{enu:r-contour-intercept-with-phi})
we cannot have that both $\psi$ and $\phi$ are equal to the function
$(0,1)\ni r\mapsto m_{0}r.$

We prove (2)(a). Noting that $\phi(r_{0})=\psi(r_{0})=s_{0}$ we have
$r_{0}\in(c,d)\cap(a,1)$ so that $a<r_{0}<d$, establishing (2)(a). 

We prove (2)(b). Assume $a=c=0$ and let $(x_{n})\subseteq(0,r_{0})$
be any strictly decreasing sequence that converges to zero. Assuming
$\phi$ does not equal the function $r\mapsto m_{0}r$, by repeatedly
applying Proposition~\ref{prop:s-contour-analysis}(\ref{enu:s-contour-intercept-with-phi}),
we notice that 
\[
\phi(x_{n+1})>\frac{\phi(x_{n})}{x_{n}}x_{n+1}
\]
which implies that $(\phi(x_{n})/x_{n})$ is strictly increasing.
Since $(\phi(x_{n})/x_{n})$ is bounded above by $1$, the limit $\lim_{r\downarrow0}\frac{\phi(r)}{r}$
exists by The Monotone Convergence Theorem and is strictly greater
than $m_{0}$ since $r_{0}^{-1}\phi(r_{0})=m_{0}$. Similarly, by
Proposition~\ref{prop:r-contour-analysis}(\ref{enu:r-contour-intercept-with-phi})
and assuming $\psi$ does not equal the function $r\mapsto m_{0}r$,
we obtain 
\[
\psi(x_{n+1})<\frac{\psi(x_{n})}{x_{n}}x_{n+1}.
\]
The sequence is strictly $(\psi(x_{n})/x_{n})$ is strictly decreasing
and bounded below by zero, hence the limit $\lim_{r\downarrow0}\frac{\psi(r)}{r}$
exists and strictly less than $m_{0}$ since $r_{0}^{-1}\psi(r_{0})=m_{0}$.
Hence $\lim_{r\downarrow0}\frac{\psi(r)}{r}\leq m_{0}\leq\lim_{r\downarrow0}\frac{\phi(r)}{r}$
and one of the inequalities must be strict, since not both $\psi$
and $\phi$ are equal to the function $(0,1)\ni r\mapsto m_{0}r$,
establishing~(2)(b).

We prove (2)(c). Assume $d=1$. By Propositions~\ref{prop:s-contour-analysis}(\ref{enu:s-contour-intercept-with-phi})
and~\ref{prop:r-contour-analysis}(\ref{enu:r-contour-intercept-with-phi})
we have
\[
\lim_{r\uparrow1}\phi(r)\leq m_{0}\leq\lim_{r\uparrow1}\psi(r).
\]
Since $\srdisjunct{\eta,\zeta}$ is true, Propositions~\ref{prop:s-contour-analysis}(\ref{enu:s-contour-intercept-with-phi})
and~\ref{prop:r-contour-analysis}(\ref{enu:r-contour-intercept-with-phi})
imply that one of these inequalities must be strict\footnote{\label{fn:strict-ineq-1}With $m_{0}:=s_{0}/r_{0}$ and any fixed
$r_{1}\in(r_{0},1)$ define $m_{1}:=\phi(r_{1})/r_{1}$. By Proposition~\ref{prop:s-contour-analysis}(\ref{enu:s-contour-g-monotone}),
if $\phi$ does not equal the function $(0,1)\ni r\mapsto m_{0}r$,
then $\phi(r_{1})<m_{0}r_{1}$ and for all $r\in(r_{1},1)$, we have
$\phi(r)<m_{1}r<m_{0}r_{1}<m_{0}r$. So that $\lim_{r\uparrow1}\phi(r)\leq m_{1}<m_{0}.$
A similar argument holds for $\psi$ through application of Proposition~\ref{prop:r-contour-analysis}(\ref{enu:r-contour-intercept-with-phi}).}. This establishes (2)(c).

We prove (2) implies (1). 

Assume $\rverticalcont{\zeta}$ is true and there exists some $r_{0}\in(a,1)$
so that $\eta\cdot\a(r_{0},\phi(r_{0}))=2\pi$ and $\zeta\cdot\b(r_{0},\phi(r_{0}))=2\pi$.
With $s_{0}:=\phi(r_{0})$, by Proposition~\ref{prop:unique-intercepts},
$(r_{0},s_{0})\in\Delta_{3}$ is the unique point for which $\eta\cdot\a(r_{0},s_{0})=2\pi$
and $\zeta\cdot\b(r_{0},s_{0})=2\pi$.

Assume $\rverticalcont{\zeta}$ is false and all the statements (2)(a),
(2)(b) or (2)(c) are true. 

If $d<1$, then Proposition~\ref{prop:r-contour-analysis}(\ref{enu:r-contour-intercept-with-phi})
implies that $\lim_{r\uparrow d}\psi(r)=d$, so that, regardless of
the value of $d$ (using (2)(c) when $d=1$), we have $\lim_{r\uparrow d}\phi(r)<\lim_{r\uparrow d}\psi(r)$.
Hence there exists some $r_{1}\in(\max\{a,c\},d)$ so that $\phi(r_{1})<\psi(r_{1})$. 

Also, Proposition~\ref{prop:s-contour-analysis}(\ref{enu:s-contour-intercept-with-phi})
implies that $\lim_{r\downarrow a}\phi(r)=a$ and Proposition~\ref{prop:r-contour-analysis}(\ref{enu:r-contour-intercept-with-phi})
implies that $\lim_{r\downarrow c}\psi(r)=0$. Therefore, if $\max\{a,c\}>0$,
then there exists some $r_{2}\in(\max\{a,c),r_{1})$ so that $\psi(r_{2})<\phi(r_{2})$.
On the other hand if $\max\{a,c\}=0$, then by (2)(b) there also exists
some $r_{2}\in(0,r_{1})$ so that $\psi(r_{2})/r_{2}<\phi(r_{2})/r_{2}$,
and hence we also have $\psi(r_{2})<\phi(r_{2})$.

Now, by the Intermediate Value Theorem, there exists some $r_{0}\in(r_{2},r_{1})$
so that $s_{0}:=\phi(r_{0})=\psi(r_{0})$ and hence $\eta\cdot\a(r_{0},s_{0})=2\pi$
and $\zeta\cdot\b(r_{0},s_{0})=2\pi$. By Proposition~\ref{prop:unique-intercepts},
we may conclude this point $(r_{0},s_{0})$ is unique.
\end{proof}

\section{\label{sec:Numerical-computation-of-upperbound-of-Pi3}Computation
of an upper bound for $|\Pi_{3}|$}

With $K$ as defined in Section~\ref{sec:Pi_3 is finite}, for any
$(\eta,\zeta)\in K$, the quantities $\lim_{r\downarrow0}\frac{\psi(r)}{r}$,
$\lim_{r\downarrow0}\frac{\phi(r)}{r}$, $\lim_{r\uparrow1}\phi(r)$,
$\lim_{r\uparrow1}\psi(r)$ etc. mentioned in Proposition~\ref{prop:intercept-necessary-and-sufficient-conditions}(2)
can easily be computed numerically to arbitrary precision. We computed
these quantities to precision $10^{-300}$ using \emph{Sympy} \cite{sympy}
in conjunction with \emph{mpmath} \cite{mpmath}. In many cases these
approximations are sufficiently accurate to determine whether or not
the inequalities in Proposition~\ref{prop:intercept-necessary-and-sufficient-conditions}(2)
are strict, and hence whether elements in $(\eta,\zeta)\in K$ determine
a unique point $(r_{0},s_{0})\in\Delta_{3}$ for which $\eta\cdot\a(r_{0},s_{0})=2\pi$
and $\zeta\cdot\b(r_{0},s_{0})=2\pi$.

The remaining cases where strict inequality between the quantities
in Proposition~\ref{prop:intercept-necessary-and-sufficient-conditions}(2)
could not be definitively determined by high precision numerical approximations,
the quantities were each determined exactly as a root of a polynomial.
We briefly describe how we determine these polynomials. In Section~\ref{sec:exact-values-for-L}
we show how we may determine bivariate polynomials $p$ and $q$ whose
solution sets necessarily contain the graphs of the functions $\phi$
and $\psi$ as yielded by Propositions~\ref{prop:s-contour-analysis}(\ref{enu:s-contour-implicit-function})
and~\ref{prop:r-contour-analysis}(\ref{enu:r-contour-implicit-function}).
The limit $\lim_{r\downarrow0}\frac{\phi(r)}{r}$ is then a solution
in $m$ of the polynomial equation $\lim_{r\to0}p(r,mr)=0$. The values
$a$ and $\lim_{r\to1}\phi(r)$ appearing in Proposition~\ref{prop:intercept-necessary-and-sufficient-conditions}(2)
are solutions of the polynomial equations $p(r,r)=0$ and $p(1,s)=0$
respectively. Similarly for $q$ and $\psi$.

Now, where strict inequality between the quantities in Proposition~\ref{prop:intercept-necessary-and-sufficient-conditions}(2)
could not be determined by numerics, they can be confirmed as equal
by being the roots of identical polynomials in all remaining cases.
Hence, by Proposition~\ref{prop:intercept-necessary-and-sufficient-conditions},
such $(\eta,\zeta)\in K$ do not admit a solution in $\Delta_{3}$
to the equations $\eta\cdot\a(r,s)=2\pi$ and $\zeta\cdot\b(r,s)=2\pi$.

Finally, the results of these computations together with Proposition~\ref{prop:intercept-necessary-and-sufficient-conditions},
allow for determining all the elements $(\eta,\zeta)\in K$ satisfying
\[
\eta\cdot\a(r_{0},s_{0})=\zeta\cdot\b(r_{0},s_{0})=2\pi
\]
for a unique $(r_{0},s_{0})\in\Delta_{3}$. This, in turn, allows
for determining an upper bound on the cardinality of the set
\[
L:=\set{(r,s)\in\Delta_{3}}{\begin{array}{c}
\exists(\eta,\zeta)\in K,\\
\eta\cdot\a(r,s)=2\pi,\\
\zeta\cdot\b(r,s)=2\pi.
\end{array}},
\]
which contains $\Pi_{3}$ (cf. Theorem~\ref{thm:set-radii-admitting-3packings-are-finite}).
This establishes the sharper bound
\[
|\Pi_{3}|\leq|L|\leq13617.
\]

We note that excluding certain (numerically approximated) elements
from $L$ as not actually being in $\Pi_{3}$ seems to be computationally
infeasible on consumer hardware. By Theorem~\ref{thm:necessary-conditions-on-r-s-tuples},
a necessary condition for $(r_{0},s_{0})\in L$ to be an element of
$\Pi_{3}$ is that there exists some $\xi\in\tuples$ for which $\onec{\xi}$
is true and $\xi\cdot\c(r_{0},s_{0})=2\pi$. The closer a point $(r_{0},s_{0})\in L$
is to the origin, the larger the search space of elements $\xi\in\tuples$
becomes for which one must verify $\onec{\xi}$ and $\xi\cdot\c(r_{0},s_{0})=2\pi$.
E.g., $(r_{1},s_{1})=(0.0000581261\ldots,0.0000125188\ldots)\in\Delta_{3}$
approximates an element\footnote{The point $(0.0000581261602\ldots,0.0000125188787\ldots)\in\Delta_{3}$
approximates the point $(r_{0},s_{0})\in L$ which is defined as the
solution of 
\[
(0,0,1,1,1,1)\cdot\a(r,s)=(1,0,4,0,2,0)\cdot\b(r,s)=2\pi.
\]
The exact values of $r_{0}$ and $s_{0}$ are as roots of the 16th
degree polynomials\begin{dmath*}
471537r^{16}-41484960r^{15}-659124096r^{14}+58464363120r^{13}+1743725080084r^{12}+17900565761408r^{11}+80565633090512r^{10}+135832773328592r^{9}-55749863701666r^{8}-312172905934624r^{7}-79130757636960r^{6}+18998456541200r^{5}+5684720044996r^{4}-232167452096r^{3}-4432749936r^{2}-23293776r+1369
\end{dmath*}and%

\begin{dmath*}
9s^{16}-2952s^{15}+297624s^{14}-9490392s^{13}+146307340s^{12}-1264707784s^{11}+6454982728s^{10}-19303597784s^{9}+30925167782s^{8}-17475748952s^{7}-13037319960s^{6}+14055271864s^{5}+4034895724s^{4}-2996664152s^{3}-1151616584s^{2}-109340424s+1369
\end{dmath*} respectively (cf. Section~\ref{sec:exact-values-for-L}).} in $L$ and by computing $\c(r_{1},s_{1})$, a naive bound on the
number of elements $\xi\in\tuples$ in the search space can be seen
to be roughly $7\times10^{21}$ which will require a considerable
length of time to sift through. Some experiments with programs written
in \emph{Cython} \cite{cython} indicate that 12 months of computation
on a modest quad-core desktop PC, utilizing all four cores, is an
extremely optimistic estimate for how long such a search might take.
Hence, a reasonably sized compute cluster is required to perform such
searches within a reasonable time. Searches might also be further
sped up by utilizing graphics processing units.

Furthermore, to confirm that an element $L$ is in $\Pi_{3}$, it
is of course required to construct a compact packing with the specified
radii.

Still, for some elements of $L$ (which are not too close to the origin)
one is able to verify within a reasonable amount of time whether or
not they satisfy the mentioned necessary condition, and ultimately,
whether they are elements of $\Pi_{3}$ by constructing packings.
We display some of them in the last section.

\section{\label{sec:exact-values-for-L}Exact computation of elements from
$L.$}

With $L$ as defined in the previous section, we describe how we may
compute exact values of elements of $L$ (and hence of elements of
$\Pi_{3}$) as roots of polynomials.

With $(\eta,\zeta)\in K$ and $(r_{0},s_{0})\in L$ satisfying $\eta\cdot\a(r_{0},s_{0})=2\pi$
and $\zeta\cdot\b(r_{0},s_{0})=2\pi$, we consider the equations 
\[
\cos(\eta\cdot\a(r,s))-1=0\quad\text{and}\quad\cos(\zeta\cdot\b(r,s))-1=0.
\]
A simple algorithm (Algorithm~\ref{alg:detrig} below) can be used
to manipulate the above system into a system of two-variable polynomial
equations, for certain $p,q\in\Z[r,s]$,
\begin{align*}
p(r,s) & =0\\
q(r,s) & =0
\end{align*}
that necessarily has $(r_{0},s_{0})$ as a solution. 

With $\phi$ and $\psi$ as yielded by Propositions~\ref{prop:s-contour-analysis}
and~\ref{prop:r-contour-analysis}, by construction, the polynomials
$p$ and $q$ will necessarily satisfy $p(r,\phi(r))=0$ and $q(r,\psi(r))=0$
for all $r$ in the respective domains of $\phi$ and $\psi$. This
observation allows for exactly computing the quantities mentioned
in Proposition~\ref{prop:intercept-necessary-and-sufficient-conditions}(2)
as discussed in Section~\ref{sec:Numerical-computation-of-upperbound-of-Pi3}.

Although Algorithm~\ref{alg:detrig} is very simple and easily implemented
in a computer algebra system, for certain values in $\tuples$ the
computation may be slow and very RAM intensive yielding large\footnote{The polynomial yielded by applying Algorithm~\ref{alg:detrig} to
the expression $\cos((1,1,12,1,1,1)\cdot\b(r,s))-1$ is 33.6 megabytes
when written unfactorized to a plain text file.} results. 

Furthermore, by computing appropriate Gr\"obner bases for the ideal
generated by $p$ and $q$ (cf. \cite[Section~2.3]{AdamsLoustaunau}
or \cite[Chapter~3]{CoxLittleOShea}) we may eliminate a variable
from each polynomial and  hence express the coordinates of $(r_{0},s_{0})\in L$
as a roots of univariate polynomials. Again, for certain inputs, computing
these Gr\"obner bases can sometimes be very RAM intensive. Some of
our computations required more than 200GB of RAM, at which point they
were halted.

We implemented Algorithm~\ref{alg:detrig} below in \emph{Sympy \cite{sympy}
}and\emph{ SymEngine \cite{SymEngine}, }and used\emph{ Singular \cite{Singular}
}for further factoring of results and for computing Gr\"obner bases.
\begin{lyxalgorithm}
{}\label{alg:detrig}
\begin{enumerate}
\item \textbf{Input:} A \textbf{\emph{given expression}}\emph{ }which is
of the form 
\[
\cos(\eta\cdot\t(r,s))-1
\]
 where $\eta\in\tuples$ and $\t\in\{\a,\b,\c\}$.\medskip
\item Expand the \textbf{\emph{given expression}} and convert all terms
in it to have a common denominator. Define the\emph{ }\textbf{\emph{partial
result}} as the numerator of this expression (Note that, by inspection
of the relevant trigonometric identities, the denominator can be seen
to be non-zero for all $(r,s)\in\Delta_{3}$, and may hence be disregarded.
Furthermore, all radicals that occur are square-roots).\medskip
\item While the\emph{ }\textbf{\emph{partial result}}\emph{ }has terms with
square-roots of expressions in variables $r$ and/or $s$ as factors,
we repeatedly do the following:\medskip
\begin{itemize}
\item Let \textbf{\emph{rad}}\emph{ }be any square-root of an expression
in $r$ and/or $s$ occurring as a factor to a term in the \textbf{\emph{partial
result}}\emph{.}\medskip
\item Let \textbf{\emph{left}}\emph{ }be the sum of all terms in the \textbf{\emph{partial
result}}\emph{ }which contains \textbf{\emph{rad}}\emph{ }as a factor\medskip
\item Let\emph{ }\textbf{\emph{right}}\emph{ }be the sum of all terms in
the \textbf{\emph{partial result}}\emph{ }which do not contain \textbf{\emph{rad}}\emph{
}as a factor.\medskip
\item Fully expand both sides of the equation \textbf{\emph{left}}\emph{$^{2}$
= }\textbf{\emph{right}}\emph{$^{2}$} and redefine the \textbf{\emph{partial
result}}\emph{ }as \textbf{\emph{left}}\emph{$^{2}$ \textendash{}
}\textbf{\emph{right}}\emph{$^{2}$}.\medskip
\end{itemize}
\item \textbf{Return:} \textbf{\emph{partial result}}\emph{.}\medskip
\end{enumerate}
\end{lyxalgorithm}

The following example displays the result of applying Algorithm~\ref{alg:detrig}. 
\begin{example}
\label{exa:s-polynomial}We apply Algorithm~\ref{alg:detrig} to
the expression 
\[
\cos((0,0,0,1,1,3)\cdot\a(r,s))-1.
\]
This expression expands to\begin{align*}
	& \frac{1}{2\left(r+s\right)^{4}\left(s+1\right)^{2}}  \times \\
	& \quad\quad
	\left(
		\begin{array}{l}
			-2 r^{4} s^{2} 
			- 4 r^{4} s 
			- 2 r^{4} 
			- 14 r^{3} s^{3} 
			- 10 r^{3} s^{2} 
			- 8 r^{3} s 
			- 30 r^{2} s^{4} 
			- 18 r^{2} s^{3} \\
			- 12 r^{2} s^{2} 
			- 18 r s^{5} 
			- 30 r s^{4}
			- 8 r s^{3}
			- 2 s^{5}
			- 2 s^{4} \\
			+ r^{2} s \sqrt{r^{2} + 2 r s} \sqrt{16 r^{2} s + 16 r s^{2} + 16 r s} 
			- 2 r s^{3} \sqrt{r^{2} + 2 r s} \sqrt{2 s + 1}	\\
			+ 3 r^{2} s \sqrt{2 s + 1} \sqrt{16 r^{2} s + 16 r s^{2} + 16 r s} 			  
			+ 10 r s^{2} \sqrt{r^{2} + 2 r s} \sqrt{2 s + 1} \\
			+ 2 r s^{2} \sqrt{r^{2} + 2 r s} \sqrt{16 r^{2} s + 16 r s^{2} + 16 r s} 
			- 6 s^{3} \sqrt{r^{2} + 2 r s} \sqrt{2 s + 1} \\
			+ 6 r s^{2} \sqrt{2 s + 1} \sqrt{16 r^{2} s + 16 r s^{2} + 16 r s} 
			- 6 s^{4} \sqrt{r^{2} + 2 r s} \sqrt{2 s + 1} \\			
			- 3 s^{3} \sqrt{r^{2} + 2 r s} \sqrt{16 r^{2} s + 16 r s^{2} + 16 r s} 
			- 2 r^{3} \sqrt{r^{2} + 2 r s} \sqrt{2 s + 1}\\
			- s^{3} \sqrt{2 s + 1} \sqrt{16 r^{2} s + 16 r s^{2} + 16 r s} 
			+ 2r^{3}s\sqrt{r^{2} + 2 r s} \sqrt{2 s + 1}  \\
			+ 6 r^{2} s^{2} \sqrt{r^{2} + 2 r s} \sqrt{2 s + 1}  
			- 2 r^{2} s \sqrt{r^{2} + 2 r s} \sqrt{2 s + 1} 
		\end{array}
	\right). 		
\end{align*}We eliminate the denominator, which is never zero on $\Delta_{3}$,
and eliminating the radicals as described in Algorithm~\ref{alg:detrig}
yields the following two-variable polynomial (factorized for the sake
of expressing it more compactly): 
\[
-16s^{2}\left(r+s\right)^{4}\left(s+1\right)^{4}\parenth{\begin{array}{l}
r^{6}s-4r^{6}+12r^{5}s^{2}-26r^{5}s-4r^{5}+54r^{4}s^{3}\\
\ \ -40r^{4}s^{2}-7r^{4}s+108r^{3}s^{4}+28r^{3}s^{3}+12r^{3}s^{2}\\
\ \ \ \ +81r^{2}s^{5}+60r^{2}s^{4}+14r^{2}s^{3}-18rs^{5}-16rs^{4}+s^{5}
\end{array}}^{2}.
\]
By construction, this polynomial has the property that, if $(r_{0},s_{0})\in\Delta_{3}$
is such that $(0,0,0,1,1,3)\cdot\a(r_{0},s_{0})=2\pi,$ then $(r_{0},s_{0})$
necessarily is also a root of the above polynomial. 
\end{example}

\section{\label{sec:Examples}Examples of compact $3$-packings}

In this section we display an arbitrary selection of compact $3$-packings.
We stress that this list is far from exhaustive.
\begin{example}\label{example1}
The values $s_{0}\approx0.299248$ and $r_{0}\approx0.438405$ are
approximations to the unique solution in $\Delta_{3}$ of the equations
\begin{align*}
(0,0,0,1,1,3)\cdot\a(r,s) & =2\pi,\\
(1,0,3,0,2,0)\cdot\b(r,s) & =2\pi,\\
(0,0,2,4,0,4)\cdot\c(r,s) & =2\pi.
\end{align*}
Exact values of $s_{0}$ and $r_{0}$ are as roots of the respective
polynomials
\[
s^{6}-54s^{5}+175s^{4}-68s^{3}+15s^{2}-6s+1
\]
 and 
\[
5r^{6}+38r^{5}+39r^{4}-28r^{3}+19r^{2}-10r+1.
\]
Figure~\ref{fig:(0, 0, 0, 1, 1, 3) (1, 0, 3, 0, 2, 0) (0, 0, 2, 4, 0, 4)}
displays a compact $3$-packing $P$ with $\rad(P)=\{s_{0},r_{0},1\}$.
\begin{figure}[H]
\fbox{\includegraphics[width=1\textwidth]{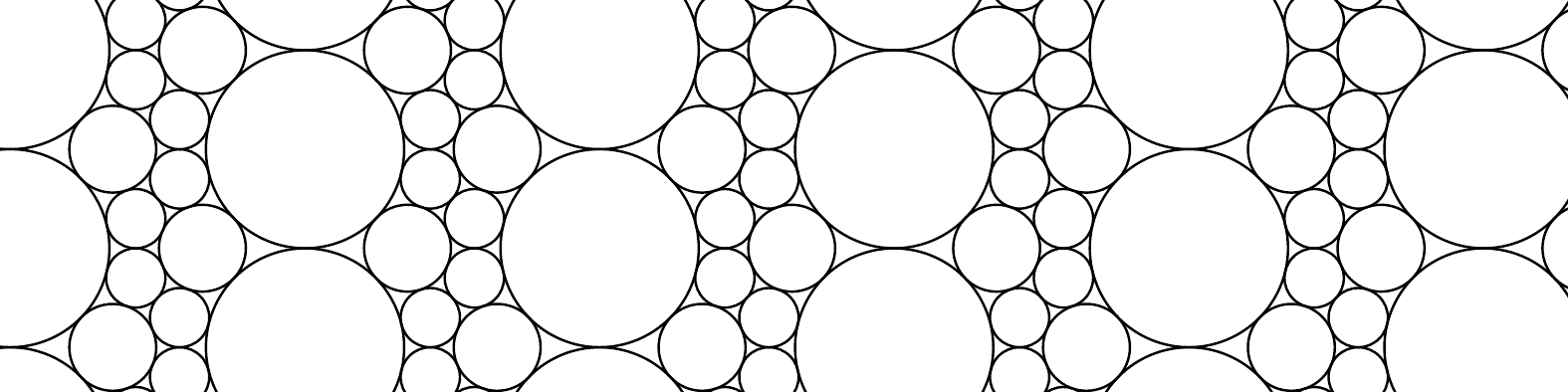}}
\caption{\label{fig:(0, 0, 0, 1, 1, 3) (1, 0, 3, 0, 2, 0) (0, 0, 2, 4, 0, 4)}
	A compact circle-packing $P$
	with $\protect\rad(P)=\{s_{0},r_{0},1\}$ where $s_{0}\approx0.299248$ and $r_{0}\approx0.438405$	
	 are roots of the polynomials as given in Example~\ref{example1}.	 
	}%
\end{figure}%
\end{example}

{}
\begin{example}\label{example2}
The values $s_{0}\approx0.468169$ and $r_{0}\approx0.822210$ are
approximations to the unique solution in $\Delta_{3}$ of the equations
\begin{align*}
(0,0,0,1,1,3)\cdot\a(r,s) & =2\pi,\\
(0,0,3,2,2,0)\cdot\b(r,s) & =2\pi,\\
(0,2,2,0,0,4)\cdot\c(r,s) & =2\pi.
\end{align*}
Exact values of $s_{0}$ and $r_{0}$ are as roots of the respective
polynomials 
\[
49s^{9}-340s^{8}+1200s^{7}-1600s^{6}-378s^{5}+560s^{4}+64s^{3}-64s^{2}-7s+4
\]
 and 
\[
2r^{9}+17r^{8}+120r^{7}+56r^{6}+60r^{5}-2r^{4}-88r^{3}-40r^{2}+2r+1.
\]
Figure~\ref{fig:(0, 0, 0, 1, 1, 3), (0, 0, 3, 2, 2, 0), (0, 2, 2, 0, 0, 4)}
displays a compact $3$-packing $P$ with $\rad(P)=\{s_{0},r_{0},1\}$.
\begin{figure}[H]
\fbox{\includegraphics[width=1\textwidth]{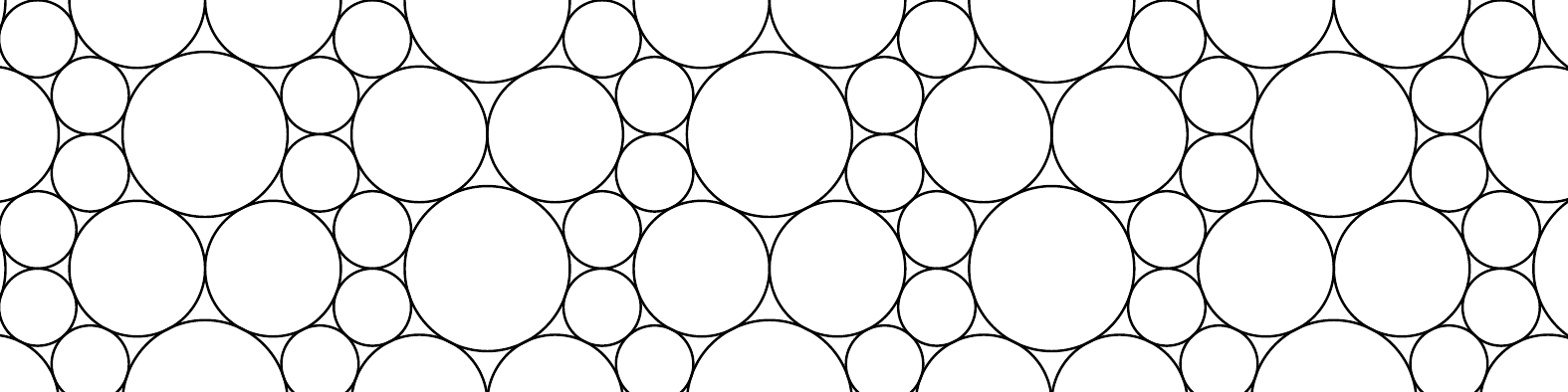}}\caption{\label{fig:(0, 0, 0, 1, 1, 3), (0, 0, 3, 2, 2, 0), (0, 2, 2, 0, 0, 4)}
	A compact circle-packing $P$
	with $\protect\rad(P)=\{s_{0},r_{0},1\}$ where $s_{0}\approx0.468169$ and $r_{0}\approx0.822210$
	 are roots of the polynomials
	 as given in Example~\ref{example2}.	 
}%
\end{figure}
\end{example}

{}
\begin{example}\label{example3}
The values $s_{0}\approx0.484497$ and $r_{0}\approx0.865150$ are
approximations to the unique solution in $\Delta_{3}$ of the equations
\begin{align*}
(0,0,0,1,1,3)\cdot\a(r,s) & =2\pi,\\
(2,0,3,0,2,0)\cdot\b(r,s) & =2\pi,\\
(0,0,1,4,0,2)\cdot\c(r,s) & =2\pi.
\end{align*}
Exact values of $s_{0}$ and $r_{0}$ are as roots of the respective
polynomials 
\begin{align*}
 & s^{11}-824s^{10}+5452s^{9}-14096s^{8}+24438s^{7}-20688s^{6}\\
 & \quad\quad+15404s^{5}-13520s^{4}-3375s^{3}+5480s^{2}+192s-512
\end{align*}
 and 
\begin{align*}
 & r^{11}+18r^{10}+132r^{9}+568r^{8}+1454r^{7}+1788r^{6}\\
 & \quad\quad+308r^{5}-680r^{4}+121r^{3}-670r^{2}-1120r+128.
\end{align*}
Figure~\ref{fig:(0, 0, 0, 1, 1, 3), (2, 0, 3, 0, 2, 0), (0, 0, 1, 4, 0, 2)}
displays a compact $3$-packing $P$ with $\rad(P)=\{s_{0},r_{0},1\}$.
\begin{figure}[H]
\fbox{\includegraphics[width=1\textwidth]{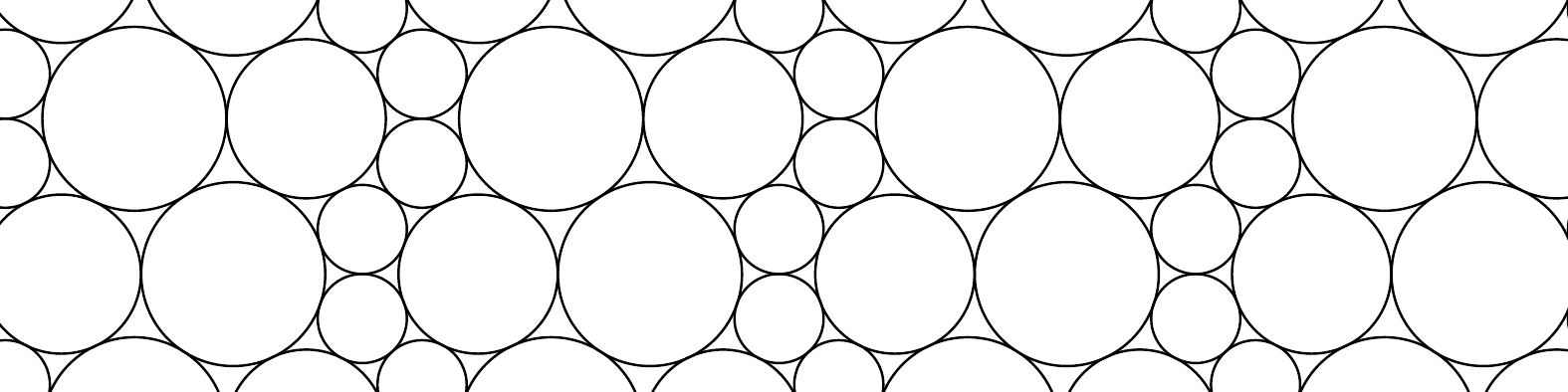}}\caption{\label{fig:(0, 0, 0, 1, 1, 3), (2, 0, 3, 0, 2, 0), (0, 0, 1, 4, 0, 2)}
A compact circle-packing $P$
	with $\protect\rad(P)=\{s_{0},r_{0},1\}$ where $s_{0}\approx0.484497$ and $r_{0}\approx0.865150$
	are roots of the polynomials
	as given in Example~\ref{example3}.	 
}%
\end{figure}
\end{example}

{}
\begin{example}\label{example4}
The values $s_{0}\approx0.275178$ and $r_{0}\approx0.948799$ are
approximations to the unique solution in $\Delta_{3}$ of the equations
\begin{align*}
(0,0,0,2,2,0)\cdot\a(r,s) & =2\pi,\\
(0,0,0,2,6,0)\cdot\b(r,s) & =2\pi,\\
(0,1,3,0,0,6)\cdot\c(r,s) & =2\pi.
\end{align*}
Exact values of $s_{0}$ and $r_{0}$ are as roots of the respective
polynomials 
\[
20s^{4}-36s^{3}+13s^{2}+6s-2
\]
 and 
\[
5r^{4}+24r^{3}+15r^{2}-38r-2.
\]
Figure~\ref{fig:(0, 0, 0, 2, 2, 0) (0, 0, 0, 2, 6, 0) (0, 1, 3, 0, 0, 6)}
displays a compact $3$-packing $P$ with $\rad(P)=\{s_{0},r_{0},1\}$.
\begin{figure}[H]
\fbox{\includegraphics[width=1\textwidth]{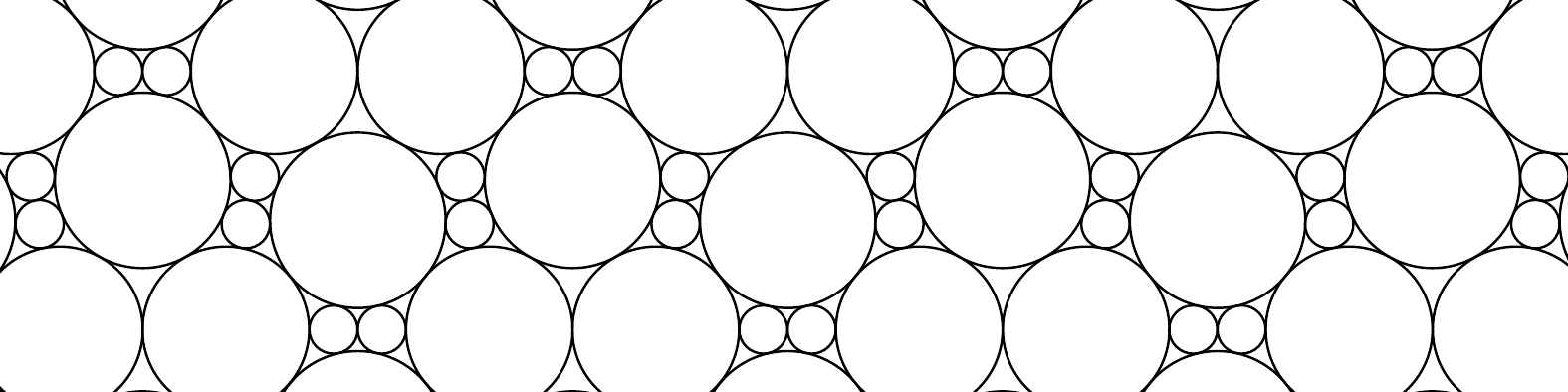}}\caption{\label{fig:(0, 0, 0, 2, 2, 0) (0, 0, 0, 2, 6, 0) (0, 1, 3, 0, 0, 6)}
A compact circle-packing $P$
	with $\protect\rad(P)=\{s_{0},r_{0},1\}$ where $s_{0}\approx0.275178$ and $r_{0}\approx0.948799$
	 are roots of the polynomials
	 as given in Example~\ref{example4}.	 
}%
\end{figure}
\end{example}

{}
\begin{example}\label{example5}
The values $s_{0}\approx0.237538$ and $r_{0}\approx0.667499$ are
approximations to the unique solution in $\Delta_{3}$ of the equations
\begin{align*}
(0,0,0,2,2,0)\cdot\a(r,s) & =2\pi,\\
(1,1,0,2,2,0)\cdot\b(r,s) & =2\pi,\\
(2,1,1,2,0,2)\cdot\c(r,s) & =2\pi.
\end{align*}
Exact values of $s_{0}$ and $r_{0}$ are as roots of the respective
polynomials 
\begin{align*}
 & 64s^{12}-704s^{11}+15792s^{10}-33536s^{9}+29964s^{8}-4540s^{7}\\
 & \quad\quad-4859s^{6}+3322s^{5}-1757s^{4}+136s^{3}+307s^{2}-102s+9
\end{align*}
 and 
\begin{align*}
 & r^{12}-4r^{11}+66r^{10}-3324r^{9}+727r^{8}+56696r^{7}+81500r^{6}\\
 & \quad\quad-29400r^{5}-46657r^{4}+332r^{3}+5314r^{2}+276r+9.
\end{align*}
Figure~\ref{fig:(0, 0, 0, 2, 2, 0) (1, 1, 0, 2, 2, 0) (2, 1, 1, 2, 0, 2)}
displays a compact $3$-packing $P$ with $\rad(P)=\{s_{0},r_{0},1\}$.
\begin{figure}[H]
\fbox{\includegraphics[width=1\textwidth]{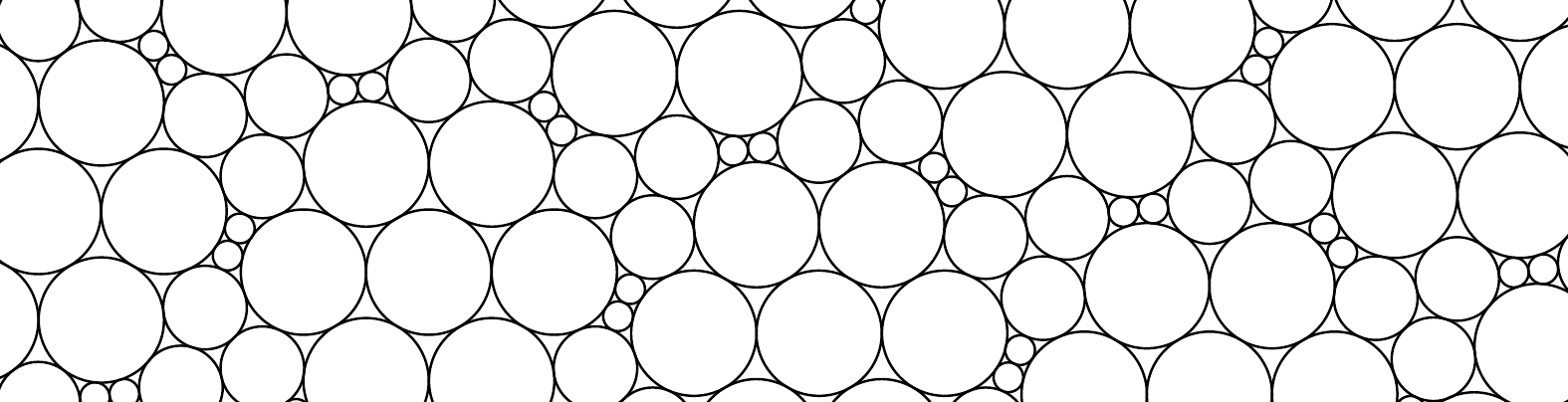}}\caption{\label{fig:(0, 0, 0, 2, 2, 0) (1, 1, 0, 2, 2, 0) (2, 1, 1, 2, 0, 2)}%
A compact circle-packing $P$
	with $\protect\rad(P)=\{s_{0},r_{0},1\}$ where  $s_{0}\approx0.237538$ and $r_{0}\approx0.667499$
	are roots of the polynomials
	as given in Example~\ref{example5}.	 
}%
\end{figure}
\end{example}